\documentclass{svjour3} 

\usepackage{cite} 
\usepackage{amsmath,amscd}
\usepackage{amssymb} 
\usepackage[retainorgcmds]{IEEEtrantools}
\usepackage{graphicx}
 \usepackage{tikz-cd}
\usepackage{comment}
\usepackage[polish, english]{babel}
\usepackage{enumerate}
\usepackage{mathtools}
\usepackage{yfonts}

\DeclarePairedDelimiter{\norm}{\lVert}{\rVert}

\DeclarePairedDelimiter\abs{\lvert}{\rvert}

\bibdata{<biblio.bib>}

\providecommand{\abs}[1]{\lvert#1\rvert} \providecommand{\norm}[1]{\lVert#1\rVert}

\begin{document}

\title{Leaves decompositions in Euclidean spaces and optimal transport of vector measures}
\titlerunning{Leaves decompositions and optimal transport}
\author{Krzysztof J. Ciosmak
\thanks{The author wishes to thank Bo'az Klartag for proposing to work on this problem and for useful discussions. The financial support of St. John's College in Oxford is gratefully acknowledged. Part of this research was completed in Fall 2017 while the author was member of the Geometric Functional Analysis and Application program at MSRI, supported by the National Science Foundation under Grant No. 1440140.}}
\institute{Krzysztof J. Ciosmak \at University of Oxford, Mathematical Institute,\\
Andrew Wiles Building, Radcliffe Observatory Quarter,\\
Woodstock Rd, Oxford OX2 6GG, United Kingdom,\\
\email{ciosmak@maths.ox.ac.uk}, \and 
University of Oxford, St John's College,\\
St Giles', Oxford OX1 3JP, United Kingdom,\\
\email{krzysztof.ciosmak@sjc.ox.ac.uk}.
}

\label{firstpage}
\maketitle

\begin{abstract}

For a given $1$-Lipschitz map $u\colon\mathbb{R}^n\to\mathbb{R}^m$ we define a partition, up to a set of Lebesgue measure zero, of $\mathbb{R}^n$ into maximal closed convex sets such that restriction of $u$ is an isometry on these sets. 

We consider a disintegration, with respect to this partition, of a log-concave measure. We prove that for almost every set of the partition of dimension $m$, the associated conditional measure is log-concave. This result is proven also in the context of the curvature-dimension condition $CD(\kappa,N)$ for weighted Riemannian manifolds. This partially confirms a conjecture of Klartag.

We provide a counterexample to another conjecture of Klartag that, given a vector measure on $\mathbb{R}^n$ with total mass zero, the conditional measures, with respect to partition obtained from a certain $1$-Lipschitz map, also have total mass zero. We develop a theory of optimal transport for vector measures and use it to answer the conjecture in the affirmative provided a certain condition is satisfied. 

\end{abstract}

\section{Introduction}

Localisation is a technique in geometry that allows to reduce $n$-dimensional problems to one-dimensional problems. It first appeared in works of Payne and Weinberger \cite{Payne} and was developed in the context of convex geometry by Gromov and Milman \cite{Gromov}, Lov\'asz and Simonovits \cite{Lovasz1} and by Kannan, Lov\'asz and Simonovits \cite{Lovasz2}. Later, Klartag \cite{Klartag} adapted the technique to the setting of weighted Riemannian manifolds satisfying curvature-dimension condition in the sense of Bakry and \'Emery  \cite{Bakry}. Subsequently, Ohta \cite{Ohta} generalised these results to Finsler manifolds and  Cavalletti and Mondino \cite{Cavalletti2, Cavalletti3} generalised them to metric measure spaces satisfying the curvature-dimension condition as defined by Sturm \cite{Sturm1, Sturm2} and by Lott and Villani \cite{Villani3}.

The purpose of this paper is to continue along the line of this research and investigate multi-dimensional analogue of the localisation technique, as proposed in \cite[Chapter 6]{Klartag}. 
In \cite{Klartag} the Monge-Kantorovich transport problem is exploited to provide a suitable partition of a given Riemannian manifold $\mathcal{M}$. Let us mention that these ideas for the norm cost function originate in the work of Sudakov \cite{Sudakov}. They allowed Ambrosio \cite{Ambrosio3} to conclude a proof of the existence of an optimal transport map. Let $u\colon\mathcal{M}\to\mathbb{R}$ be a $1$-Lipschitz that maximises the integral
\begin{equation}\label{eqn:integ}
\int_{\mathcal{M}} v d(\mu-\nu)
\end{equation}
among all $1$-Lipschitz maps $v\colon\mathcal{M}\to\mathbb{R}$. Here $\mu,\nu$ are two Borel probability measures on $\mathcal{M}$. Then partition arises as geodesics of maximal growth of $u$, i.e. the integral curves of the gradient of $u$.

In what follows, we  consider finite dimensional linear spaces equipped with Euclidean norm, unless specified otherwise, and $1$-Lipschitz maps $u\colon\mathbb{R}^n\to\mathbb{R}^m$, $m\leq n$. We define a partition, up to Lebesgue measure zero, of $\mathbb{R}^n$, associated to such a map and prove its basic properties. The sets of the partition are the maximal sets $\mathcal{S}$ such that the restriction of $u$ to $\mathcal{S}$ is an isometry, i.e. preserves the Euclidean distance. Each such set we shall call a \emph{leaf} of $u$. We prove that each leaf of $u$ is closed and convex, hence it has a well-defined dimension. Suppose now that $(\mathbb{R}^n,d,\mu)$ is a weighted Riemannian manifold, satisfying curvature-dimension condition $CD(\kappa,N)$ for some $\kappa\in\mathbb{R}$ and $N\in (-\infty,1)\cup [n,\infty]$, see Section \ref{sec:curv} for definitions. Here $d$ denotes the Euclidean metric on $\mathbb{R}^n$ and $\mu$ is a Borel finite measure on $\mathbb{R}^n$. We prove that for almost every leaf $\mathcal{S}$ of dimension $m$, the weighted Riemannian manifold $(\mathrm{int}\mathcal{S},d,\mu_{\mathcal{S}})$ also satisfies $CD(\kappa,N)$ condition. Here $\mu_{\mathcal{S}}$ denote the conditional measures of $\mu$ with respect to the partition into leaves of $u$.

Below we denote by $\lambda$ the Lebesgue measure on $\mathbb{R}^n$ and $CC(\mathbb{R}^n)$ denotes the set of non-empty closed, convex subsets of $\mathbb{R}^n$, equipped with Wijsman topology, see \cite{Wijsman}.

\begin{theorem}\label{thm:d}
Let $u\colon\mathbb{R}^n\to\mathbb{R}^m$ be a $1$-Lipschitz map with respect to the Euclidean norms. Then there exists a map $\mathcal{S}\colon\mathbb{R}^n\to CC(\mathbb{R}^n)$ such that for $\lambda$-almost every $x\in\mathbb{R}^n$ the set $\mathcal{S}(x)$ is a maximal closed convex set in $\mathbb{R}^n$ such that $u|_{\mathcal{S}(x)}$ is an isometry.
Suppose that $\mu$ is a Borel finite measure on $\mathbb{R}^n$  such that $(\mathbb{R}^n,d,\mu)$ is a weighted Riemannian manifold satisfying $CD(\kappa,N)$ for some $\kappa\in\mathbb{R}$ and $N\in (-\infty,1)\cup [n,\infty]$. Then there exist a Borel measure on $CC(\mathbb{R}^n)$ and Borel measures $\mu_{\mathcal{S}}$ such that 
\begin{equation*}
\mathcal{S}\mapsto \mu_{\mathcal{S}}(A)\text{ is }\nu\text{-measurable for any Borel set }A\subset\mathbb{R}^n
\end{equation*}
and for $\nu$-almost every $\mathcal{S}$ we have $\mu_{\mathcal{S}}(\mathcal{S}^c)=0$,
and for any $A\subset\mathbb{R}^n$
\begin{equation*}
\mu(A)=\int_{CC(\mathbb{R}^m)} \mu_{\mathcal{S}}(A)d\nu(\mathcal{S}).
\end{equation*}
Moreover, for $\nu$-almost every leaf $\mathcal{S}$ of dimension $m$, weighted Riemannian manifold $(\mathrm{int}\mathcal{S},d,\mu_{\mathcal{S}})$ satisfies the curvature-dimension condition $CD(\kappa,N)$. 
\end{theorem}

Here we denote by $\mathcal{S}$ the map and a set. This should not lead to any ambiguity, as for $\lambda$-almost every $y\in\mathbb{R}^n$ such that $y\in\mathcal{S}(x)$ for some $x\in\mathbb{R}^n$ we have $\mathcal{S}(x)=\mathcal{S}(y)$, see Corollary \ref{col:unique}.

The theorem provides a partial positive answer to a conjecture of Klartag posed in \cite[Chapter 6]{Klartag}, where it is conjectured that the above theorem holds true also for $\nu$-almost every leaf $\mathcal{S}$ of lower dimensions.

Note that the absolute continuity of the conditional measures with respect to a partition into convex sets may fail to be true. Indeed, as proved in \cite{Ambrosio1} and in \cite{Larman}, there exists a measurable partition, up to a set of Lebesgue measure zero, of the unit cube in $\mathbb{R}^3$ into pairwise disjoint line segments such that the conditional measures are Dirac measures.

The result enriches the knowledge of regularity properties of Lipschitz maps. For $m=1$ such regularity was necessary to prove the existence of optimal transport map in the Monge-Kantorovich problem (see \cite{Sudakov}, \cite{Ambrosio3}, \cite{Caffarelli}). We refer the reader to \cite{Villani1}, \cite{Villani2} and \cite{Kolesnikov} for an account on the optimal transport problem.

The possible applications of the result are in the localisation or dimensional reduction arguments, where the disintegration is an effective tool. A similar result to ours in case $m=1$ has been used to derive new proofs and generalisations of isoperimetric inequality, Poincar\'e's inequality and others to the setting of metric measure spaces satisfying curvature bounds. We refer the reader to \cite{Klartag}, \cite{Cavalletti2}, \cite{Cavalletti3}, \cite{Ohta}.

The proof relies on the area formula and Fubini's theorem and is based on a previous work \cite{Caffarelli}. See also \cite{Ambrosio3} and \cite{Feldman} for similar approach to the Monge-Kantorovich problem.
Another tool that we use is the Wijsman topology \cite{Wijsman} on the closed subsets of $\mathbb{R}^n$ which makes it a Polish space, so we may apply disintegration theorem.

Let us note that there exists a different method of proving the absolute continuity of the conditional measures in the case $m=1$. This method is present in this context in \cite{Caravenna1}. It is also applied in \cite{Bianchini2} and in \cite{Caravenna2}. In \cite{Bianchini1}, it was used to complete the idea of a proof proposed by Sudakov in \cite{Sudakov} of existence of an optimal Monge's map with norm cost.
The Fubini's theorem and a clever application of the Thales's theorem are the core of the idea. The absolute continuity of the conditional measures is not proved directly, but, instead, it is shown that the measures of the orthogonal sections are absolutely continuous with respect to each other.

%

Suppose now that we are given a Borel probability measure $\mu$ on $\mathbb{R}^n$ absolutely continuous with respect to the Lebesgue measure such that
\begin{equation*}
\int_{\mathbb{R}^n}fd\mu=0
\end{equation*}
for some integrable function $f\colon\mathbb{R}^n\to\mathbb{R}^m$ such that
\begin{equation*}
\int_{\mathbb{R}^n}\norm{f(x)}\norm{x}d\mu(x)<\infty.
\end{equation*}
Let $u\colon\mathbb{R}^n\to\mathbb{R}^m$ be a $1$-Lipschitz map such that
\begin{equation}\label{eqn:maxx}
\int_{\mathbb{R}^n}\langle u,f\rangle d\mu=\sup\Big\{\int_{\mathbb{R}^n}\langle v,f\rangle d\mu\big| v\colon\mathbb{R}^n\to\mathbb{R}^m \text{ is }1\text{- Lipschitz}\Big\}.
\end{equation}
In \cite[Chapter 6]{Klartag} it is conjectured that  
\begin{equation}\label{eqn:mass}
\int_{\mathbb{R}^n}fd\mu_{\mathcal{S}}=0\text{ for }\nu\text{- almost every }\mathcal{S}\in CC(\mathbb{R}^n),
\end{equation}
where $\{\mu_{\mathcal{S}}|\mathcal{S}\in CC(\mathbb{R}^n)\}$ is disintegration of $\mu$ with respect to the leaves of $u$, and $\nu$ is the push forward of $\mu$ with respect to the map $\mathcal{S}$.

We provide a counterexample to this conjecture. Moreover we show that such statement fails to be true even if we replace the set of $1$-Lipschitz maps in (\ref{eqn:maxx}) by any locally uniformly closed subset of $1$-Lipschitz maps with respect to \emph{any norm} on $\mathbb{R}^n$ and \emph{any strictly convex norm} on $\mathbb{R}^m$, unless the set of maps is trivial, i.e. consisting only of isometries.
Note that the outline of a proof of the conjecture suggested in \cite{Klartag} has a gap, as follows by \cite{Ciosmak}.

We develop a theory of optimal transport of vector measures and establish its basic properties. 
We show, among others, that for a given vector measure, there may be no optimal transport. However, if an optimal transport exists and has certain absolute continuity properties, then we prove that the conjecture of Klartag holds true.

Let us mention the existence of another approach to optimal transport of vector measures that differs from ours developed by Chen, Georgiou, Tannenbaum, Tyu, Li and Osher (see \cite{Chen1, Chen2}).
%

\section{Outline of the article}

Here we describe the structure of the paper. In Section \ref{sec:partition} we provide a careful definition of the partition associated to a $1$-Lipschitz map. What will follow in the latter sections is the existence of the map $\mathcal{S}\colon\mathbb{R}^n\to CC(\mathbb{R}^n)$ satisfying the properties of Theorem \ref{thm:d}. We prove that certain components of $u$ are  differentiable on certain leaves. Moreover we investigate the regularity of the derivative on the leaves of dimension $m$ and provide an interesting strengthening of $1$-Lipschitz property of $u$, see Lemma \ref{lem:important} and Remark \ref{rmk:strength}.

In Section \ref{sec:varia} we define a Lipschitz change of variables on certain sets, called \emph{clusters}, that will allow us to use area formula and then Fubini's theorem to prove the regularity properties of the conditional measures. Here we provide significantly simpler proofs than the proofs in \cite{Caffarelli}, mainly thanks to Lemma \ref{lem:important} and Corollary \ref{col:strength}.

In Section \ref{sec:measur} we prove measurability properties of the partition, which will allow us to show the map $\mathcal{S}\colon\mathbb{R}^n\to CC(\mathbb{R}^n)$ is measurable with respect to the Wijsman topology on $CC(\mathbb{R}^n)$. We also prove that the set of boundaries of leaves of maximal dimension is a Borel set  of the Lebsegue measure zero.

In Section \ref{sec:disin} we provide a part of a proof of Theorem \ref{thm:d}. 

In Section \ref{sec:transport} we provide a definition of optimal transport of $\mathbb{R}^m$-valued vector measures on a metric space. We prove basic theorems about the optimal transport of vector measures and show that it is a convex dual to the problem (\ref{eqn:maxx}). Using this theory we provide a positive answer the aforementioned conjecture, provided that there exists an optimal transport such that the marginals of its total variation are absolutely continuous, see Theorem \ref{thm:balance}.

In Section \ref{sec:counter} we assume that $m>1$ and we provide an aforementioned counterexample which show that in general the so-called mass balance condition (\ref{eqn:mass}) does not hold true. Let $\mathcal{F}$ be any subset of $1$-Lipschitz maps that is locally uniformly closed. We prove that (\ref{eqn:mass}) fails to be true, when the maximisation problem (\ref{eqn:maxx}) is replaced by 
\begin{equation}
\sup\Big\{\int_{\mathbb{R}^n}\langle v,f\rangle d\mu\big| v\in\mathcal{F}\Big\},
\end{equation}
unless $\mathcal{F}$ is trivial in the sense that any $u$ that attains the above supremum is an isometry. This is shown for any norm on $\mathbb{R}^n$ and any strictly convex norm on $\mathbb{R}^m$.

In Section \ref{sec:curv} we prove that the conditional measures $\mu_{\mathcal{S}}$ have densities such that the weighted Riemannian manifolds $(\mathrm{int}\mathcal{S},d,\mu_{\mathcal{S}})$ satisfy the curvature-dimension condition.

\section{Partition and its regularity}\label{sec:partition}

If $A\subset \mathbb{R}^n$ let us denote by $\mathrm{Conv}(A)$ the \emph{convex hull} of $A$, i.e. the set 
\begin{equation*}
\Big\{\sum_{i=1}^k \lambda_i x_i\mid k\in\mathbb{N},\lambda_1,\dotsc,\lambda_k\geq 0, \sum_{i=1}^k\lambda_i=1,x_1,\dotsc,x_k\in A\Big\}.
\end{equation*}
We define the \emph{affine hull} $\mathrm{Aff}(A)$ of a set $A\subset\mathbb{R}^n$ to be 
\begin{equation*}
\Big\{\sum_{i=1}^k \lambda_i x_i\mid k\in\mathbb{N},\lambda_1,\dotsc,\lambda_k\in\mathbb{R}, \sum_{i=1}^k\lambda_i=1,x_1,\dotsc,x_k\in A\Big\}.
\end{equation*}

\begin{lemma}\label{lem:coordinates2}
Let $z_1,\dotsc,z_k\in\mathbb{R}^n$. 
Let $x\in \mathbb{R}^n$ and $y\in \mathrm{Conv}(z_1,\dotsc,z_k)$. Suppose that 
\begin{equation*}
\norm{x-z_i}\leq\norm{y-z_i},
\end{equation*}
for $i=1,\dotsc, k$. Then $x=y$.
\end{lemma}
\begin{proof}
Denote 
\begin{equation*}
\mathrm{Conv}(z_1,\dotsc,z_k)=Z.
\end{equation*}
We have
\begin{equation*}
\norm{x}^2+\norm{z_i}^2-2\langle x,z_i\rangle\leq \norm{y}^2+\norm{z_i}^2-2\langle y,z_i\rangle
\end{equation*}
for all $i=1,\dotsc,k$. Hence, for these $i$'s, we have 
\begin{equation*}
\norm{x}^2-\norm{y}^2\leq2\langle x-y,z_i\rangle.
\end{equation*}
Thus, adding up these inequalities multiplied by non-negative coefficients that sum up to one, we get
\begin{equation*}
\norm{x}^2-\norm{y}^2\leq2\langle x-y,z\rangle
\end{equation*}
for all $z\in Z$. Then, putting $z=y$, we obtain
\begin{equation*}
\norm{x}^2-\norm{y}^2\leq2\langle x,y\rangle-2\norm{y}^2, 
\end{equation*}
i.e. $\norm{x-y}^2\leq 0$.
\end{proof}

\begin{definition}
Let $u\colon \mathbb{R}^n\to\mathbb{R}^m$ be a $1$-Lipschitz function. A set $\mathcal{S}\subset \mathbb{R}^n$ is called a \emph{leaf} of $u$ if $u|_{\mathcal{S}}$ is an isometry and for any $y\notin \mathcal{S}$ there exists $x\in\mathcal{S}$ such that $\norm{u(y)-u(x)}<\norm{y-x}$.
\end{definition}

In other words, $\mathcal{S}$ is a leaf if it is a maximal set, with respect to the order induced by inclusion, such that $u|_{\mathcal{S}}$ is an isometry.

\begin{definition}
If $C\subset\mathbb{R}^n$ is a convex set, then we shall call the \emph{tangent space} of $C$ the linear space $\mathrm{Aff}(C)-\mathrm{Aff}(C)$.
We shall call the \emph{relative interior} of $C$ the relative interior with respect to the topology of $\mathrm{Aff}(C)$.
\end{definition}

\begin{lemma}\label{lem:uniq}
Let $\mathcal{S}\subset \mathbb{R}^n$ be an arbitrary subset. Let $u\colon \mathcal{S}\to \mathbb{R}^m$ be an isometry. Then there exists a unique $1$-Lipschitz function $\tilde{u}\colon \mathrm{Conv}(\mathcal{S})\to \mathbb{R}^m$ such that $\tilde{u}|_{\mathcal{S}}=u$. Moreover $\tilde{u}$ is an isometry.
\end{lemma}
\begin{proof}
Take any point $z\in \mathcal{S}$ such that 
\begin{equation*}
z=\sum_{i=1}^k t_iz_i
\end{equation*}
for some non-negative real numbers $t_1,\dotsc,t_k$ that sum up to one and some points $z_1,\dotsc, z_k\in \mathcal{S}$. We claim that 
\begin{equation}\label{eqn:affinity}
u(z)=\sum_{i=1}^k t_iu(z_i).
\end{equation} We have
\begin{equation*}
\norm{u(z)-u(z_i)}=\norm{z-z_i}.
\end{equation*} 
Moreover, by polarisation formula, $u$ preserves the scalar product, i.e. for all points $r,s,t\in \mathcal{S}$
\begin{equation*}
\begin{aligned}
&\langle u(r)-u(s), u(t)-u(s)\rangle =\\ &=\frac 12 \big( \norm{u(r)-u(s)}^2+\norm{u(t)-u(s)}^2-\norm{u(r)-u(t)}^2\big)=
\langle r-s,t-s\rangle.
\end{aligned}
\end{equation*}
Hence
\begin{equation}\label{eqn:isometric}
\begin{aligned}
&\Big\lVert\sum_{i=1}^k t_iu(z_i)-u(z_l)\Big\rVert^2=\Big\lVert\sum_{i=1}^k t_i(u(z_i)-u(z_l))\Big\rVert^2=\\
&=\sum_{i,j=1}^k t_it_j \langle u(z_i)-u(z_l),u(z_j)-u(z_l)\rangle=\sum_{i,j=1}^k t_it_j \langle z_i-z_l,z_j-z_l\rangle=\\
&=\Big\lVert\sum_{i=1}^k t_iz_i-z_l\Big\rVert^2=\norm{z-z_l}^2.
\end{aligned}
\end{equation}
Thus, by Lemma \ref{lem:coordinates2}, equation (\ref{eqn:affinity}) holds true.
We may now extend $u$ to $\mathrm{Conv}(\mathcal{S})$ by affinity. That is, if $x_1,\dotsc,x_r\in \mathcal{S}$ are any points in general position, i.e. vectors $(x_i-x_1)_{i=1}^r$ are linearly independent, and $s_1,\dotsc,s_r$ are any non-negative real numbers that sum up to $1$, we set
\begin{equation*}
\tilde{u}\bigg(\sum_{i=1}^r s_ix_i\bigg)=\sum_{i=1}^r s_iu(x_i).
\end{equation*}
Function $\tilde{u}$ defined in such a way is affine. Hence, there exist a linear map $T\colon V\to\mathbb{R}^m$ defined on the tangent space $V$ of $\mathrm{Conv}(\mathcal{S})$  and a vector $b\in\mathbb{R}^m$ such that
\begin{equation*}
\tilde{u}(y)=T(y-y_0)+b
\end{equation*}
for any $y\in\mathrm{Conv}(\mathcal{S})$ and some $y_0\in\mathrm{Conv}(\mathcal{S})$.
We claim that $\tilde{u}$ is an isometry. For this, it is enough to check that $T$ is isometric on a set $(r_i-r_0)_{i=1}^l$, where $r_0,\dotsc,r_l\in \mathcal{S}$ are such that 
\begin{equation*}
\mathrm{span}(r_i-r_0)_{i=1}^l=V.
\end{equation*}
The latter follows from the assumption that $u$ is isometric on $\mathcal{S}$.

Suppose now that we have another $1$-Lipschitz extension $v\colon \mathrm{Conv}(\mathcal{S})\to\mathbb{R}^m$. To prove that $v=\tilde{u}$ it is enough to show that $v$ is affine. Choose non-negative real numbers $s_1,\dotsc,s_r$ summing up to $1$ and any points $x_1,\dotsc,x_r\in \mathcal{S}$.
Then, by $1$-Lipschitzness and by the fact that $v$ is isometric on $\mathcal{S}$, we get, as in (\ref{eqn:isometric}),
\begin{equation*}
\Big\lVert v\Big(\sum_{i=1}^rs_ix_i\Big)-v(x_j)\Big\rVert\leq \Big\lVert \sum_{i=1}^rs_ix_i-x_j\Big\rVert =\Big\lVert\sum_{i=1}^rs_iv(x_i)-v(x_j)\Big\rVert.
\end{equation*}
By Lemma \ref{lem:coordinates2} we see that 
\begin{equation*}
v\bigg(\sum_{i=1}^rs_ix_i\bigg)=\sum_{i=1}^rs_iv(x_i).
\end{equation*}
Any point in $\mathrm{Conv}(\mathcal{S})$ is a convex combination of points $\mathcal{S}$, so the condition of affinity of $v$ also holds for any convex combination of points in $\mathrm{Conv}(\mathcal{S})$.
\end{proof}

\begin{corollary}
Any leaf $\mathcal{S}$ of $u$ is a closed convex set and $u|_{\mathcal{S}}$ is an affine isometry.
\end{corollary}

Let $\mathcal{S}$ be a leaf of $u$. Let $P$ denote the orthogonal projection of $\mathbb{R}^n$ onto the tangent space $V$ of $\mathcal{S}$. Let 
\begin{equation*}
T\colon V\to\mathbb{R}^m
\end{equation*}
be a linear isometry such that 
\begin{equation*}
u(y)=T(y-y_0)+b
\end{equation*}
for any $y\in\mathcal{S}$, some $y_0\in\mathcal{S}$ and some $b\in\mathbb{R}^m$. Let $Q$ denote the orthogonal projection of $\mathbb{R}^m$ onto $T(V)$.

Below by $\mathrm{int}\mathcal{S}$, $\mathrm{cl}\mathcal{S}$, $\partial\mathcal{S}$ we understand the relative interior, the relative closure and the relative boundary of $\mathcal{S}$ respectively.

\begin{lemma}\label{lem:important}
Let $u\colon \mathbb{R}^n\to\mathbb{R}^m$ be a $1$-Lipschitz map. Let $\mathcal{S}_1,\mathcal{S}_2$ be two leaves of $u$. Let $V_1,V_2$ be their respective tangent spaces and let $P_1,P_2$ be orthogonal projections onto $V_1,V_2$ respectively. Let $T_1,T_2$ be isometric maps such that 
\begin{equation*}
u(x)-u(y)=T_i(x-y)\text{ for all }x,y\in\mathcal{S}_i, i=1,2.
\end{equation*}
Let $x_i\in\mathcal{S}_i$ and $\sigma_i=\mathrm{dist}(x_i,\partial \mathcal{S}_i)$ for $i=1,2$. Then
\begin{equation*}
2\sigma_1\sigma_2\norm{P_1P_2-P_1T_1^*T_2P_2}\leq \norm{x_1-x_2}^2-\norm{u(x_1)-u(x_2)}^2,
\end{equation*}
and for $i=1,2$
\begin{equation*}
2\sigma_i\norm{P_iT_i^*(u(x_1)-u(x_2))-P_i(x_1-x_2)}\leq\norm{x_1-x_2}^2-\norm{u(x_1)-u(x_2)}^2.
\end{equation*}
\end{lemma}
\begin{proof}
Let $y_i\in\mathcal{S}_i$ for $i=1,2$. Let $v_i=y_i-x_i$ for $i=1,2$. Then we may write
\begin{equation*} 
u(y_1)-u(y_2)=u(x_1)-u(x_2)+T_1v_1-T_2v_2.
\end{equation*}
Hence $\norm{u(y_1)-u(y_2)}^2$ is equal to 
\begin{equation*}
\norm{u(x_1)-u(x_2)}^2+\norm{v_1}^2+\norm{v_2}^2+2\langle u(x_1)-u(x_2),T_1v_1-T_2v_2\rangle-2\langle T_1v_1,T_2v_2\rangle.
\end{equation*}
We also have
\begin{equation*}
y_1-y_2=x_1-x_2+v_1-v_2,
\end{equation*}
yielding
\begin{equation*}
\norm{y_1-y_2}^2=\norm{x_1-x_2}^2+\norm{v_1}^2+\norm{v_2}^2+2\langle x_1-x_2,v_1-v_2\rangle-2\langle v_1,v_2\rangle.
\end{equation*}
As $u$ is $1$-Lipschitz, $\norm{u(y_1)-u(y_2)}\leq\norm{y_1-y_2}$. By the two identities above we get therefore that 
\begin{equation*}
2\langle v_1,v_2\rangle-2\langle T_1v_1,T_2v_2\rangle+2\langle u(x_1)-u(x_2),T_1v_1-T_2v_2\rangle-2\langle x_1-x_2,v_1-v_2\rangle
\end{equation*}
is bounded above by 
\begin{equation*}
\norm{x_1-x_2}^2-\norm{u(x_1)-u(x_2)}^2
\end{equation*}
Suppose that $\sigma_1, \sigma_2$ are positive. As $y_1,y_2$ were arbitrary points of $\mathcal{S}_1,\mathcal{S}_2$ respectively, the above inequality holds true for any $v_1\in V_1$ and $v_2\in V_2$ of norm at most $\sigma_1$ and $\sigma_2$ respectively. If we add two such inequalities with $v_1,v_2$ changed to $-v_1,-v_2$ then we get that 
\begin{equation*}
2\langle v_1,v_2\rangle-2\langle T_1v_1,T_2v_2\rangle\leq \norm{x_1-x_2}^2-\norm{u(x_1)-u(x_2)}^2.
\end{equation*}
Equivalently for any $w_1,w_2\in\mathbb{R}^n$ of norm at most one we have
\begin{equation*}
\sigma_1\sigma_2\big\langle w_1,(P_1P_2-P_1T_1^*T_2P_2)w_2\big\rangle\leq \norm{x_1-x_2}^2-\norm{u(x_1)-u(x_2)}^2.
\end{equation*}
Taking supremum over all $w_1,w_2\in\mathbb{R}^n$ of norm at most one yields the first desired inequality.
For the next inequalities, we assume that $\sigma_2>0$ and we put $v_1=0$ to get that
\begin{equation*}
-2\langle u(x_1)-u(x_2),T_2v_2\rangle+2\langle x_1-x_2,v_2\rangle\leq \norm{x_1-x_2}^2-\norm{u(x_1)-u(x_2)}^2.
\end{equation*}
Analogously for $v_2=0$ and $\sigma_1>0$
\begin{equation*}
2\langle u(x_1)-u(x_2),T_1v_1\rangle-2\langle x_1-x_2,v_1\rangle\leq \norm{x_1-x_2}^2-\norm{u(x_1)-u(x_2)}^2.
\end{equation*}
Hence
\begin{equation*}
\sigma_2\bigg\langle\big( P_2T_2^*\big(u(x_1)-u(x_2)\big)-P_2(x_1-x_2)\big),w_2\bigg\rangle\leq \norm{x_1-x_2}^2-\norm{u(x_1)-u(x_2)}^2
\end{equation*}
and
\begin{equation*}
\sigma_1\bigg\langle\big( P_1T_1^*\big(u(x_1)-u(x_2)\big)-P_1(x_1-x_2)\big),w_1\bigg\rangle\leq \norm{x_1-x_2}^2-\norm{u(x_1)-u(x_2)}^2.
\end{equation*}
Taking suprema over $w_1,w_2$ in the unit ball of $\mathbb{R}^n$ yields the desired results.
\end{proof}

\begin{remark}\label{rmk:strength}
Lemma \ref{lem:important} tells us that if $x_1,x_2$ belong to relative interiors of leaves $\mathcal{S}_1,\mathcal{S}_2$ respectively, then the $1$-Lipschitzness of map $u\colon\mathbb{R}^n\to\mathbb{R}^m$ is strengthened to the condition that
\begin{equation*}
\norm{u(x_1)-u(x_2)}^2+2\sigma_1\sigma_2\norm{P_1P_2-P_1T_1^*T_2P_2}\leq \norm{x_1-x_2}^2.
\end{equation*} 
\end{remark}

\begin{corollary}\label{col:strength}
Let $u\colon\mathbb{R}^n\to\mathbb{R}^m$ be a $1$-Lipschitz map. Let $x_i\in\mathrm{int}\mathcal{S}_i$ belong to the relative interior of leaf $\mathcal{S}_i$ of $u$ of dimension $m$, for $i=1,2$.  Then
\begin{equation*}
\norm{Du(x_1)-Du(x_2)}\leq \sqrt{\frac{\norm{x_1-x_2}^2-\norm{u(x_1)-u(x_2)}^2}{2\sigma_1\sigma_2}}.
\end{equation*}
Here $\sigma_i=\mathrm{dist}(\partial\mathcal{S}_i,x_i)$ for $i=1,2$.
\end{corollary}
\begin{proof}
As the dimensions of leaves are equal to $m$, the respective projections $Q_i$ onto the images of $T_i$ are identities and $Du(x_i)=T_iP_i$, for $i=1,2$. Note that $Q_i=T_iP_i(T_iP_i)^*$ for $i=1,2$. Inferring as in Lemma \ref{lem:important} we get that
\begin{equation*}
\big|\norm{P_1v_1-P_2v_2}^2-\norm{T_1P_1v_1-T_2P_2v_2}^2\big|\leq\frac{\norm{x_1-x_2}^2-\norm{u(x_1)-u(x_2)}^2}{2\sigma_1\sigma_2}
\end{equation*}
for all $v_1,v_2\in\mathbb{R}^n$ of norm at most one. Taking $v_1=(T_1P_1)^*v$ and $v_2=(T_2P_2)^*v$ for some unit vector $v\in\mathbb{R}^m$ yields
\begin{equation*}
\norm{(T_1P_1)^*-(T_2P_2)^*}^2\leq \frac{\norm{x_1-x_2}^2-\norm{u(x_1)-u(x_2)}^2}{2\sigma_1\sigma_2}
\end{equation*}
since $Q_1=Q_2$. Taking the square root concludes the proof.
\end{proof}

%

\begin{lemma}\label{lem:diff}
Let $\mathcal{S}$ be a leaf of a $1$-Lipschitz map $u\colon \mathbb{R}^n\to\mathbb{R}^m$. Then $Qu$ is differentiable in the relative interior of $\mathcal{S}$. Moreover, if $z_0$ belongs to the relative interior of $\mathcal{S}$, then
\begin{equation*}
DQu(z_0)=TP.
\end{equation*}
If $u$ is differentiable in $z_0$ for some $z_0\in\mathcal{S}$, then
\begin{equation*}
QDu(z_0)=TP.
\end{equation*}
\end{lemma}
\begin{proof}
By Lemma \ref{lem:important} we see that 
\begin{equation*}
2\sigma \norm{Q (u(z_1)-u(z_0))-TP(z_1-z_0)}\leq \norm{z_1-z_0}^2-\norm{u(z_1)-u(z_0)}^2.
\end{equation*}
for all $z_0\in\mathcal{S}$ and $z_1\in\mathbb{R}^n$. Here $\sigma=\mathrm{dist}(z_0,\partial\mathcal{S})$. Hence if $\sigma>0$ we obtain that
\begin{equation*}
\limsup_{z_1\to z_0} \frac{\norm{Q(u(z_1)-u(z_0)-TP(z_1-z_0)}}{\norm{z_1-z_0}}\leq \limsup_{z_1\to z_0} \frac{\norm{z_1-z_0}}{\sigma}=0.
\end{equation*}
This yields the asserted differentiability. Now, suppose that $u$ is differentiable at $z_0\in\mathcal{S}$.  Inferring as in the proof of Lemma \ref{lem:important} we see that for all $z_2\in \mathcal{S}$ we have
\begin{equation*}
2\big\langle T^*(u(z_1)-u(z_0))-(z_1-z_0),(z_2-z_0)\big\rangle\leq \norm{z_1-z_0}^2-\norm{u(z_1)-u(z_0)}^2.
\end{equation*}
Take any $w\in\mathbb{R}^n$ and let $z_1=z_0+tw$. Let $t$ tend to zero. Then the above inequality implies that
\begin{equation*}
\langle T^* Du(z_0)w-w,z_2-z_0\rangle\leq 0.
\end{equation*}
As this holds true for any $w\in\mathbb{R}^n$, applying this inequality to $-w$, we infer that the above inequality is an equality, i.e.
\begin{equation*}
\langle T^* Du(z_0)w-w,z_2-z_0\rangle= 0.
\end{equation*}
If follows that for all $v\in \mathrm{span}\{z_2-z_0|z_2\in\mathcal{S}\}=V$
\begin{equation*}
\langle T^* Du(z_0)w-w,v\rangle= 0,
\end{equation*}
and consequently $\langle Q Du(z_0)w- TPw, Tv\rangle=0$. The assertion follows.
\end{proof}

\begin{corollary}\label{col:diff}
Suppose that $\mathcal{S}$ is of dimension $m$. Then $u$ is differentiable in the relative interior of $\mathcal{S}$.
\end{corollary}

\begin{lemma}\label{lem:boundary}
Let $\mathcal{S}_1, \mathcal{S}_2$ be two distinct leaves of a $1$-Lipschitz map $u\colon \mathbb{R}^n\to\mathbb{R}^m$. Then, 
\begin{equation*}
\mathcal{S}_1\cap\mathcal{S}_2\subset\partial\mathcal{S}_1\cap\partial\mathcal{S}_2.
\end{equation*}
\end{lemma}
\begin{proof}

We shall first show that there is no point belonging to $ \mathrm{int}\mathcal{S}_1\cap  \mathcal{S}_2$. 
For this, suppose that $x_0\in\mathrm{int}\mathcal{S}_1\cap\mathcal{S}_2$. Let $x_1\in\mathcal{S}_1$ and $x_2\in\mathcal{S}_2$. There exists isometries $T_1$ and $T_2$ on the tangent spaces $V_1$ and $V_2$ of $\mathcal{S}_1$ and $\mathcal{S}_2$ respectively such that
\begin{equation*}
u(x_1)-u(x_0)=T_1(x_1-x_0)\text{ and }u(x_2)-u(x_0)=T_2(x_2-x_0).
\end{equation*}
We may write
\begin{align*}
&\norm{x_1-x_0}^2+\norm{x_2-x_0}^2-2\langle T_1(x_1-x_0),T_2(x_2-x_0)\rangle=\norm{u(x_1)-u(x_2)}^2 \leq\\
&\leq \norm{x_1-x_2}^2=\norm{x_1-x_0}^2+\norm{x_2-x_0}^2-2\langle x_1-x_0,x_2-x_0\rangle.
\end{align*}
Hence
\begin{equation*}
\langle x_1-x_0,x_2-x_0\rangle\leq \langle T_1(x_1-x_0),T_2(x_2-x_0)\rangle.
\end{equation*}
As $x_0\in\mathrm{int}\mathcal{S}_1$ and the inequality holds true for all $x_1\in\mathcal{S}_1$, we actually have equality above for $x_1$ sufficiently close to $x_0$. It follows that for all $v_1\in V_1$ and $v_2\in V_2$,
\begin{equation}\label{eqn:eqiso}
\langle v_1,v_2\rangle=\langle T_1v_2,T_2v_2\rangle.
\end{equation}
Hence there exists an isometry $S\colon V_1+ V_2\to\mathbb{R}^m$ that extends both $T_1$ and $T_2$. 
Indeed, define a linear map
\begin{equation*}
S\colon V_1+V_2 \to\mathbb{R}^m
\end{equation*}
by the formula
\begin{equation*}
S(v_1+v_2+v_3)=T_1(v_1)+T_2(v_2)
\end{equation*}
where
\begin{equation*}
v_1\in V_1\cap V_2^{\perp}, v_2\in V_1\cap V_2.
\end{equation*}
We claim that $S$ is a well-defined isometry. Indeed, by (\ref{eqn:eqiso}) and by orthogonality we see that if $v_2\in V_1\cap V_2$, then 
\begin{equation*}
\norm{v_2}^2=\langle T_1v_2,T_2v_2\rangle \text{, so }T_1v_2=T_2v_2.
\end{equation*}
We have
\begin{equation*}
\norm{S(v_1+v_2)}^2=\norm{v_1}^2+\norm{v_2}^2+2\langle T_1v_1,T_2v_2\rangle=\norm{v_1+v_2}^2.
\end{equation*}
Moreover, by definition $S$ is an extension of both $T_1$ and $T_2$.

Define an affine map $v\colon x_0+V_1 +V_2\to\mathbb{R}^m$ by the formula 
\begin{equation*}
v(x)=S(x-x_0)+b.
\end{equation*}
Then $v|_{\mathcal{S}_1}=u$ and $v|_{\mathcal{S}_2}=u$.

Choose any points $x\in\mathcal{S}_1$ and $y\in\mathcal{S}_2$. Then 
\begin{equation*}
\norm{u(x)-u(y)}=\norm{v(x)-v(y)}=\norm{S(x-y)}=\norm{x-y}.
\end{equation*}
Thus $u$ is isometric on $\mathcal{S}_1\cup\mathcal{S}_2$. By maximality $\mathcal{S}_1=\mathcal{S}_1\cup\mathcal{S}_2=\mathcal{S}_2$, contradicting the distinctness of the two leaves. 
Hence
\begin{equation*}
\mathcal{S}_1\cap\mathcal{S}_2\subset \partial \mathcal{S}_1\cap \mathcal{S}_2.
\end{equation*}
Repeating the above argument with $\mathcal{S}_1$ and $\mathcal{S}_2$ interchanged, we see that
\begin{equation*}
\mathcal{S}_1\cap\mathcal{S}_2\subset \big(\partial \mathcal{S}_1\cap \mathcal{S}_2\big) \cap \big(\partial \mathcal{S}_2\cap \mathcal{S}_1\big)=\partial \mathcal{S}_1\cap\partial\mathcal{S}_2.
\end{equation*}
\end{proof}

\begin{remark}
We may proceed in the first part of the above proof alternatively. Namely, let $x_0\in\mathcal{S}_1\cap \mathcal{S}_2$. 
Then Lemma \ref{lem:diff} implies that $Q_1u$ is differentiable at $x_0$ with the derivative given by
\begin{equation*}
DQ_1u(x_0)=T_1P_1,
\end{equation*}
where $T_1$ is an isometry such that $u(x)=T_1(x-x_0)+b$ for all $x\in\mathcal{S}_1$, $P_1$ is the orthogonal projection onto the tangent space $V_1$ of $\mathcal{S}_1$ and $Q_1$ is the orthogonal projection onto $\mathrm{im}T_1$. In other words
\begin{equation}\label{eqn:limit}
\lim_{x\to x_0}\frac{Q_1u(x)-Q_1u(x_0)-T_1P_1(x-x_0)}{\norm{x-x_0}}=0.
\end{equation}
For $x\in\mathcal{S}_2$ we may write
\begin{equation*}
u(x)=T_2(x-x_0)+b
\end{equation*}
for an isometry $T_2$. Let $V_2$ be the tangent space of $\mathcal{S}_2$. If $x\in\mathcal{S}_2$, then
\begin{equation}\label{eqn:difference}
\frac{Q_1u(x)-Q_1u(x_0)-T_1P_1(x-x_0)}{\norm{x-x_0}}=(Q_1T_2-T_1P_1)\bigg(\frac{x-x_0}{\norm{x-x_0}}\bigg).
\end{equation}
Observe that if $x_1\in \mathrm{int}\mathcal{S}_2$, then, as $x-x_1=x-x_0-(x_1-x_0)$,
\begin{equation}\label{eqn:v}
V_2=\mathrm{span}\{x-x_1|x\in \mathcal{S}_2\}\subset\mathrm{span}\{x-x_0|x\in\mathcal{S}_2\}\subset V_2.
\end{equation}
Let $x\in\mathcal{S}_2$. For $t\in [0,1]$ let
\begin{equation*}
x_t=x_0+t(x-x_0).
\end{equation*}
By convexity of leaves, $x_t\in\mathcal{S}_2$. Observe also that 
\begin{equation*}
\lim_{t\to 0}x_t=x_0.
\end{equation*}
It follows by (\ref{eqn:limit}), (\ref{eqn:difference}) and by (\ref{eqn:v}) that 
\begin{equation}\label{eqn:equ}
Q_1T_2v=T_1P_1v\text{ for all }v\in V_2.
\end{equation}
It follows that for $v_1\in V_1$ and $v_2\in V_2$
\begin{equation*}
\langle T_1v_1,T_2v_2\rangle=\langle T_1v_1,Q_1T_2v_2\rangle=\langle T_1v_1,T_1P_1v_2\rangle=\langle v_1,v_2\rangle.
\end{equation*}
We complete the proof as before.
\end{remark}

\begin{corollary}\label{col:notdif}
If $z_0\in\mathbb{R}^n$ belongs to at least two distinct leaves of a $1$-Lipschitz mapping $u\colon\mathbb{R}^n\to\mathbb{R}^m$ then $u$ is not differentiable at $z_0$.
\end{corollary}
\begin{proof}
Clearly, any zero dimensional leaf does not intersect any other leaf. Hence, $z_0$ belongs to two distinct leaves $\mathcal{S}_1,\mathcal{S}_2$ of non-empty relative interiors. Suppose that $u$ is differentiable at $z_0$.
Lemma \ref{lem:diff} tells us that
\begin{equation*}
QDu(z_0)=TP,
\end{equation*}
where $T$ is an isometry such $u(z)-u(z_0)=T(z-z_0)$ for all $z\in\mathcal{S}_1$, $P$ is the orthogonal projection onto the tangent space of $\mathcal{S}_1$ and $Q$ is the orthogonal projection onto $\mathrm{im}T$. Arguing as in Lemma \ref{lem:boundary}, we infer that $\mathcal{S}_1=\mathcal{S}_2$. This contradiction completes the proof.
\end{proof}

\begin{definition}
The set of points belonging to at least two distinct leaves of a $1$-Lipschitz function $u\colon\mathbb{R}^n\to\mathbb{R}^m$ we shall denote by $B(u)$.
\end{definition}

\begin{corollary}\label{col:unique}
For any $1$-Lipschitz function $u\colon\mathbb{R}^n\to\mathbb{R}^m$ the set $B(u)$ is of Lebesgue measure zero.
\end{corollary}
\begin{proof}
Corollary \ref{col:notdif} implies that $B(u)$ is contained in the set of non-differentiability of $u$. Rademacher's theorem (see e.g. \cite{Federer}) states that the latter is of Lebesgue measure zero.
\end{proof}

\section{Lipschitz change of variables}\label{sec:varia}

Let us recall a lemma taken from \cite[\S 3.2.9]{Federer}.

\begin{lemma}\label{lem:coo}
Let $u\colon\mathbb{R}^n\to\mathbb{R}^m$ be a continuous function. Then the set
\begin{equation*}
\{x\in\mathbb{R}^n|u\text{ is differentiable at }x\text{ and }Du(x)\text{ has maximal rank}\}
\end{equation*}
admits a countable Borel covering $(G_i)_{i=1}^{\infty}$  such that for any $i\in\mathbb{N}$ there exist an orthogonal projection $p\colon\mathbb{R}^n\to\mathbb{R}^{n-m}$ and Lipschitz maps
\begin{equation*}
w\colon\mathbb{R}^n\to\mathbb{R}^m\times\mathbb{R}^{n-m}\text{, }v\colon\mathbb{R}^m\times\mathbb{R}^{n-m}\to\mathbb{R}^n
\end{equation*}
such that
\begin{equation*}
w(x)=(u(x),p(x))\text{ and }v(w(x))=x\text{ for all }x\in G_i.
\end{equation*}
\end{lemma}

\begin{lemma}\label{lem:cover}
Let $u\colon\mathbb{R}^n\to\mathbb{R}^m$ be a Lipschitz function, $p\in\mathbb{R}^m$ and let
\begin{equation*}
S_p=\{x\in\mathbb{R}^n|u(x)=p\}
\end{equation*}
be the level set. Then the set
\begin{equation*}
S_p\cap\{x\in\mathbb{R}^n| u\text{ is differentiable at }x\text{ and }Du(x)\text{ has maximal rank}\}
\end{equation*}
has a countable Borel covering $(S_p^i)_{i=1}^{\infty}$ of bounded sets such that for all $i\in\mathbb{N}$ there exist Lipschitz functions $w\colon\mathbb{R}^n\to\mathbb{R}^{n-m}$ and $v\colon\mathbb{R}^{n-m}\to\mathbb{R}^n$ satisfying
\begin{equation*}
v(w(x))=x \text{ for all }x\in S_p^i.
\end{equation*}
\end{lemma}
\begin{proof}
We apply the above lemma and obtain a countable covering consisting of Borel sets $G_i$, orthogonal projections $\pi_i\colon\mathbb{R}^n\to\mathbb{R}^{n-m}$ and Lipschitz maps 
\begin{equation*}
w_i\colon\mathbb{R}^n\to\mathbb{R}^m\times\mathbb{R}^{n-m}\text{, }v_i\colon\mathbb{R}^m\times\mathbb{R}^{n-m}\to\mathbb{R}^n
\end{equation*}
such that
\begin{equation*}
w_i(x)=(u(x),\pi_i(x))\text{ and }v_i(w_i(x))=x\text{ for all }x\in G_i.
\end{equation*}
The sets $G_i\cap S_p$ form a countable Borel covering of $S_p$. For any $i\in\mathbb{N}$ define 
\begin{equation*}
w\colon\mathbb{R}^n\to\mathbb{R}^{n-m}\text{ and }
v\colon \mathbb{R}^{n-m}\to\mathbb{R}^n
\end{equation*}
by $w=\pi\circ w_i$, where $\pi\colon \mathbb{R}^m\times\mathbb{R}^{n-m}\to \mathbb{R}^{n-m}$ is the projection on the second variable, and $v(x)=v_i(p,x)$ for $x\in\mathbb{R}^{n-m}$.
\end{proof}

Choose a countable dense set $Q$ in $\mathbb{R}^m$.

\begin{definition}
Let $p\in Q$. Let $u\colon\mathbb{R}^n\to\mathbb{R}^m$ be a $1$-Lipschitz function and let $(S_p^i)_{i=1}^{\infty}$ be the Borel cover of Lemma \ref{lem:cover} associated to the level set 
\begin{equation*}
S_p=\{x\in\mathbb{R}^n|u(x)=p\}.
\end{equation*}
For each $i,j\in\mathbb{N}$ let the \emph{cluster}
\begin{equation*}
T_{pij}
\end{equation*}
denote the union of all $m$-dimensional leaves $\mathcal{S}$ of $u$ which intersect $S_p^i$ and for which the point of intersection $z\in S_p^i$ is separated from the boundary of the leaf by distance at least $1/j$. 
Denote by 
\begin{equation*}
\mathrm{int}T_{pij}
\end{equation*}
the union of the interiors of all $m$-dimensional leaves $\mathcal{S}$ of $u$ as above.
\end{definition}

\begin{lemma}\label{lem:cluster}
The union of all $m$-dimensional leaves is covered by the clusters 
\begin{equation*}
(T_{pij})_{p\in Q,i,j\in\mathbb{N}}.
\end{equation*}
Moreover for each $m$-dimensional leaf $\mathcal{S}$ and each cluster $T_{pij}$ either
\begin{equation*}
\mathrm{int} \mathcal{S}\cap T_{pij}=\emptyset\text{ or }\mathrm{int} \mathcal{S}\subset T_{pij}.
\end{equation*}
\end{lemma}
\begin{proof}
Let $\mathcal{S}$ be a $m$-dimensional leaf of $u$. Then $u$, if restricted to $\mathcal{S}$, is an isometry onto a subset of $\mathbb{R}^m$ with non-empty interior. Thus, there exists $p\in Q\cap \mathrm{int} u(\mathcal{S})$. In particular $\mathcal{S}\cap S_p\neq \emptyset$. The point $x$ in the intersection belongs to one of the covering sets $S_p^i$ of Lemma \ref{lem:cover} and lies in a positive distance from the boundary of the leaf, so $\mathcal{S}\subset T_{pij}$ for some $j\in\mathbb{N}$. 
If the interior of some other leaf $\mathrm{int}\mathcal{S}$ intersects one of the leaves  comprising the cluster $T_{pij}$, then Lemma \ref{lem:boundary} implies that they are equal and hence $\mathcal{S}\subset T_{pij}$. This completes the proof.
\end{proof}

\begin{lemma}\label{lem:efge}
Each cluster $T_{pij}\subset\mathbb{R}^n$ admits a map
\begin{equation*}
G\colon\mathrm{int} T_{pij}\to\mathbb{R}^{n-m}\times\mathbb{R}^m
\end{equation*}
and its inverse
\begin{equation*}
F\colon G( \mathrm{int}T_{pij}) \to \mathrm{int}T_{pij}
\end{equation*}
such that:
\begin{enumerate}[i)]
\item\label{i:lambda} for each $\lambda>0$ and  $\rho>0$, $G$ is a Lipschitz map on the set
\begin{equation*}
T^{\lambda,\rho}_{pij}=\bigg\{x\in \mathrm{int}T_{pij}\big| \mathrm{dist}(x,\partial \mathcal{S}(x))>\lambda,\norm{u(x)-u(z)}\leq \rho \bigg\};
\end{equation*}
here $\mathcal{S}(x)$ is the unique leaf of $u$ such that $x\in\mathcal{S}(x)$ and $z\in \mathcal{S}(x)$ is the unique point in $\mathcal{S}(x)$ such that $u(z)=p$,
\item for each $\rho>0$ $F$ is Lipschitz on the set $G(T_{pij}^{0,\rho})$,
\item $F(G(x))=x$ for each $x\in\mathrm{int} T_{pij}$,
\item if a leaf $\mathcal{S}\subset T_{pij}$ intersects $S_p^i$ at a point $z$, then each interior point $x\in \mathrm{int}\mathcal{S}$ of the leaf satisfies
\begin{equation}\label{eqn:gie}
G(x)=(w(z),u(x)-u(z)),
\end{equation}
where $w\colon\mathbb{R}^n\to\mathbb{R}^{n-m}$ is the map from Lemma \ref{lem:cover}.
\end{enumerate}
\end{lemma}
\begin{proof}
Lemma \ref{lem:boundary} shows that the relative interiors of leaves do not intersect any other leaf. Moreover $u$ is an isometry on each leaf. Therefore, every point $x\in\mathrm{int} T_{pij}$ belongs to a unique leaf and each leaf intersects the level set $S_p$ in a single point $z\in S_p^i$. It follows that (\ref{eqn:gie}) defines a map 
\begin{equation*}
G\colon\mathrm{int} T_{pij}\to\mathbb{R}^{n-m}\times\mathbb{R}^m,
\end{equation*}
on the cluster $\mathrm{int}T_{pij}$. Let $(a,b)\in G(\mathrm{int}T_{pij})$ and let $v$ be the map parametrising $S_p^i$ from Lemma \ref{lem:cover}. Then $v(a)\in S_p^i$ belongs to a relative interior of some leaf $\mathcal{S}$ and lies in a distance at least $1/j$ from the relative boundary of the leaf. Define
\begin{equation*}
F(a,b)=v(a)+Du(v(a))^*(b).
\end{equation*}
Let $x\in \mathrm{int}T_{pij}$ belong to a leaf $\mathcal{S}$ that intersects $S_p$ at a point $z$. Then 
\begin{equation*}
v(w(z))=z
\end{equation*}
and there exists an isometry $T$ such that $u(s_1)-u(s_2)=T(s_1-s_2)$ for all $s_1,s_2\in\mathcal{S}$ and $Du(z)=TP$, where $P$ is the orthogonal projection onto the tangent space of $\mathcal{S}$. We infer that
\begin{equation*}
F(G(x))=F(w(z),u(x)-u(z))=z+PT^*T(x-z)=x.
\end{equation*}
We shall now prove that for $\rho>0$, the mapping $F$ is Lipschitz on  $G(T_{pij}^{0,\rho})$. Define
\begin{equation}\label{eqn:lambda}
\Lambda=\big\{a\in\mathbb{R}^{n-m}\big| (a,0)\in G(T_{pij}^{0,\rho})\big\}.
\end{equation}
We first claim that 
\begin{equation*}
\Lambda\ni a\mapsto Du(v(a))^*\in \mathbb{R}^{n\times m}
\end{equation*}
is a Lipschitz function. Recall that $v(a)\in S_p^i$ is in a distance at least $1/j$ from the relative boundary of a leaf $\mathcal{S}$ that contains $v(a)$. Thus, by Corollary \ref{col:strength} and Lemma \ref{lem:cover}, we infer that for $a,a'\in\Lambda$
\begin{equation*}
\norm{Du(v(a))^*-Du(v(a'))^*}\leq j \norm{v(a)-v(a')}\leq Cj\norm{a-a'}.
\end{equation*}
If $(a,b)\in G(T_{pij}^{0,\rho})$, then $\norm{b}\leq \rho$. Thus $F$ is Lipschitz on $G(T_{pij}^{0,\rho})$.

It remains to prove assertion \ref{i:lambda}) of the lemma. Let $\lambda>0$ and $\rho>0$. We shall first show that the derivative $Du$ is Lipschitz on $T_{pij}^{\lambda,\rho}$. This immediately follows by Corollary \ref{col:strength}.

Let now $x,x'\in T_{pij}^{\lambda,\rho}$ belong to the leaves $\mathcal{S}$ and $\mathcal{S}'$ respectively. By the definition (\ref{eqn:gie}) to prove $1$-Lipschitzness of $G$ it is enough to show that
\begin{equation*}
\norm{w(z)-w(z')}\leq C\norm{x-x'}
\end{equation*}
for some constant $C$. As $w$ is Lipschitz map it is enought to prove that $\norm{z-z'}$ is bounded by a constant times $\norm{x-x'}$. Note that
\begin{equation*}
z=x+Du(x)^*(u(z)-u(x))\text{ and }z'=x'+Du(x')^*(u(z')-u(x')).
\end{equation*}
Thus
\begin{equation*}
\norm{z-z'}\leq \norm{x-x'}+\Big\rVert Du(x)^*(u(z)-u(x))-Du(x')^*(u(z')-u(x'))\Big\lVert.
\end{equation*}
Now, taking into account that $u(z)=u(z')=p$ and writing the latter summand as
\begin{equation*}
\Big\lVert\big(Du(x)^*-Du(x')^*\big)\big(u(z)-u(x)\big)+Du(x')^*(u(x')-u(x))\Big\rVert
\end{equation*}
we may bound it by 
\begin{equation*}
\frac{\rho}{\lambda }\norm{x-x'}+\norm{x-x'}.
\end{equation*}
This concludes the proof that $G$ is Lipschitz on $T_{pij}^{\lambda,\rho}$ and completes the proof of the theorem.
\end{proof}
%


\section{Measurability}\label{sec:measur}

Below $G_{n,k}$ denotes the space of all $k$-dimensional subspaces of $\mathbb{R}^n$. For $V\in G_{n,k}$ and $W\in G_{m,k}$ we denote by $O(V,W)$ the set of all isometries on $V$ with values in $W$ and by $P_V\colon \mathbb{R}^n\to\mathbb{R}^n$ the orthogonal projection onto $V$. Then $G_{n,k}$ is a compact if equipped with the metric given by
\begin{equation*}
d(V,V')=\norm{P_V-P_{V'}},
\end{equation*}
for $V,V'\in G_{n,k}$. Here $\norm{\cdot}$ denotes the operator norm with respect to the Euclidean norm on $\mathbb{R}^n$.

\begin{definition}
For $k\in\{1,\dotsc,m\}$ define $\alpha_k\colon \mathbb{R}^n\to \mathbb{R}\cup\{\infty\}$ by the formula
\begin{align*}
\alpha_k(x)=\sup\Big\{\epsilon\geq 0\big|\exists_{V\in G_{n,k}} \exists_{W\in G_{m,k}}\exists_{T\in O(V,W)}&\forall_{y\in (x+V)\cap B(x,\epsilon)}\\
& u(x)-u(y)=T(x-y) \Big\}, 
\end{align*}
where $B(x,\epsilon)=\{y\in\mathbb{R}^n|\norm{x-y}< \epsilon\}$. Define $\alpha_{m+1}\colon\mathbb{R}^n\to\mathbb{R}$ by $\alpha_{m+1}(x)=0$ for all $x\in\mathbb{R}^n$.
\end{definition}

\begin{lemma}\label{lem:alpha}
For any $k\in\{1,\dotsc,m\}$ the functions $\alpha_k\colon \mathbb{R}^n\to\mathbb{R}\cup\{\infty\}$ are upper semicontinuous.
\end{lemma}
\begin{proof}
Choose a sequence $(x_l)_{l=1}^{\infty}$ that converges to $x_0$ such that there exists a limit
\begin{equation*}
\alpha_k=\lim_{l\to\infty}\alpha_k(x_l).
\end{equation*}
We need to show that $\alpha_k\leq \alpha_k(x_0)$. Suppose first that $\alpha_k<\infty$. We may assume that $\alpha_k(x_l)\in\mathbb{R}$ for each $l\in\mathbb{N}$. From the definition of $\alpha_k(x_l)$ it follows that there exist 
\begin{equation*}
V_l\in G_{n,k}, W_l\in G_{m,k}\text{ and }T_l\in O(V_l,W_l)
\end{equation*}
such that for all $y\in (x_l+V_l)\cap B(x_l,(1-\frac1l)\alpha_k(x_l))$ we have
\begin{equation*}
u(x_l)-u(y)=T_l(x_l-y).
\end{equation*}
By compactness of $G_{n,k}$ and of $G_{m,k}$ we may assume that the sequences of $V_l$ and $W_l$ are convergent to some $V_0\in G_{n,k}$ and $W_0\in G_{m,k}$ and that 
\begin{equation*}
T_lP_{V_l}\text{ converges to }T_0P_{V_0}, 
\end{equation*}
where $T_0\in O(V_0,W_0)$. Indeed, let $S_l=T_lP_{V_l}$ and $R_l=T_l^{-1}P_{W_l}$. Choosing a convergent subsequences from $(S_l)_{l=1}^{\infty}$ and from $(R_l)_{l=1}^{\infty}$, we may assume that there exists $S_0,R_0$ such that
\begin{equation*}
R_0S_0=P_{V_0}\text{ and }S_0R_0=P_{W_0}.
\end{equation*}
Hence
\begin{equation*}
S_0P_{V_0}=P_{W_0}S_0\text{ and }R_0P_{W_0}=P_{V_0}R_0.
\end{equation*}
It follows that $S_0\colon V_0\to W_0$ and $R_0\colon W_0\to V_0$ are mutual reciprocals. Moreover, they are isometric. Indeed, for any $v,w\in\mathbb{R}^n$, we have
\begin{equation*}
\langle S_0P_{V_0}v, S_0P_{V_0}w\rangle =\lim_{l\to\infty} \langle S_lP_{V_l}v,S_lP_{V_l}w\rangle =\lim_{l\to\infty}\langle P_{V_l}v,P_{V_l}w\rangle=\langle P_{V_0}v,P_{W_0}w\rangle
\end{equation*}
Thus, putting $T_0$ to be $S_0$ restricted to $V_0$, we have proven the claim.

Choose any $v_0\in V_0$ of norm $\norm{v_0}< \alpha_k$. Then, by the definition of metric on $G_{n,k}$, the sequence $P_{V_l}v_0$ converges to $v_0$ Moreover, for sufficiently large $l$, 
\begin{equation*}
x_l+P_{V_l}v_0\in (x_l+V_l)\cap B\big(x_l,\big(1-1/l\big)\alpha_k(x_l)\big).
\end{equation*}
Thus
\begin{equation*}
u(x_l)-u(x_l+P_{V_l}v_0)=-T_lP_{V_l}v_0.
\end{equation*}
Passing to the limits we obtain
\begin{equation*}
u(x_0)-u(x_0+v_0)=-T_0v_0.
\end{equation*}
It follows that $\alpha_k(x_0)\geq \alpha_k$. The proof is complete if $\alpha_k$ is finite. Suppose now that $\alpha_k$ is infinite. Assume again that $\alpha_k(x_l)\in\mathbb{R}$ for each $l\in\mathbb{N}$ and that $\alpha_k(x_l)$ converges to infinity monotonically. Then there exist $V_l, W_l$ and $T_l$ as before, such that $V_l$ converges to $V_0$, $W_l$ converges to $W_0$ and $T_lP_{V_l}$ converges to $T_0P_{V_0}$. Taking any $v_0\in V_0$ of norm at most $l\in\mathbb{N}$ we may show that
\begin{equation*}
u(x_0)-u(x_0+v_0)=-T_0v_0.
\end{equation*}
Hence $\alpha_k(x_0)\geq l$ for each $l\in\mathbb{N}$ and thus $\alpha_k(x_0)=\infty$.
\end{proof}

Below we shall denote the unit ball by $B^n=\{x\in\mathbb{R}^n|\norm{x}\leq 1\}$.

\begin{definition}
For $k\in\{1,\dotsc,m\}$ define $\beta_k\colon \mathbb{R}^n\to \mathbb{R}$ by the formula
\begin{align*}
\beta_k(x)=\sup\Big\{\epsilon\geq 0\big|\exists_{\mathcal{C}\in C_{n,k}(\epsilon)}\exists_{W\in G_{m,k}}\exists_{T\in O(V_{\mathcal{C}},W)}&\forall_{y\in (x+\mathcal{C})\cap B(x,\epsilon)} \\
&u(x)-u(y)=T(x-y) \Big\}, 
\end{align*}
where $B(x,\epsilon)=\{y\in\mathbb{R}^n|\norm{x-y}< \epsilon\}$ and $V_{\mathcal{C}}=\mathrm{span}(\mathcal{C})$ and $C_{n,k}(\epsilon)$ is the set of all convex cones $\mathcal{C}$ in $\mathbb{R}^n$ of dimension $k$ such that 
\begin{equation*}
\lambda_k(\mathcal{C}\cap S^{n-1})\geq \epsilon^k.
\end{equation*}
Here $\lambda_k$ is the Lebesgue measure on the $k$-dimensional ball 
\begin{equation*}
V_{\mathcal{C}}\cap \{x\in\mathbb{R}^n|\norm{x}\leq 1\}.
\end{equation*}
Define $\beta_{m+1}\colon \mathbb{R}^n\to\mathbb{R}$ by $\beta_{m+1}(x)=0$ for all $x\in\mathbb{R}^n$.
\end{definition}

\begin{lemma}
For any $k\in\{1,\dotsc,m\}$ the function $\beta_k\colon \mathbb{R}^n\to\mathbb{R}$ is upper semicontinuous.
\end{lemma}
\begin{proof}
Choose a sequence $(x_l)_{l=1}^{\infty}$ that converges to $x_0$ and such that there exists a limit
\begin{equation*}
\beta_k=\lim_{l\to\infty}\beta_k(x_l).
\end{equation*}
We need to show that $\beta_k\leq \beta_k(x_0)$. Observe that $\beta_k<\infty$, as $\lambda_k$ is a finite measure. It follows from the definition of $\beta_k(x_l)$ that there exist 
\begin{equation*}
\mathcal{C}_l\in C_{n,k}\big(\big(1-1/l\big)\beta_k(x_l)\big), W_l\in G_{m,k} \text{ and }T_l\in O(V_{\mathcal{C}_i},W_l)
\end{equation*}
such that for all $y\in (x_l+\mathcal{C}_l)\cap B(x_l,(1-1/l\beta_k(x_l))$
\begin{equation*}
u(x_l)-u(y)=T_l(x_l-y).
\end{equation*}
Consider the sets $K_l=\mathcal{C}_l\cap B^n$. These are compact, convex sets. Taking a subsequence, we may assume that there is a compact, convex set $K_0\subset B^n$ such that $K_l$ converges to $K_0$ in the Hausdorff metric. Moreover (see  \cite{Beer}),
\begin{equation*}
\lambda_k(K_0)\geq \beta_k^k.
\end{equation*}
Let 
\begin{equation*}
\mathcal{C}_0=\big\{x\in\mathbb{R}^n\big| x=\lambda y\text{ for some }\lambda\geq 0, y\in K_0\big\}.
\end{equation*}
Then $\mathcal{C}_0\in C_{n,k}(\beta_k)$. Passing to a subsequence, we may assume that $V_{\mathcal{C}_l}$ converges to some $V_0\in G_{n,k}$. We claim now that $V_{\mathcal{C}_l}$ converges to $V_{\mathcal{C}_0}$. Choose any $v_0\in V_{\mathcal{C}_0}$. Then there exist real numbers $\lambda_1,\dotsc,\lambda_k$ and $c_1,\dotsc,c_k\in K_0$ such that
\begin{equation*}
v_0=\sum_{j=1}^k\lambda_j c_j.
\end{equation*}
By the convergence in the Hausdorff metric we infer that there exist $(c_{j,l})_{l=1}^{\infty}$, $c_{j,l}\in K_l$, such that
\begin{equation*}
\lim_{l\to\infty}c_{j,l}=c_j.
\end{equation*}
Let 
\begin{equation*}
v_l=\sum_{j=1}^k\lambda_j c_{j,l}.
\end{equation*}
Then $\lim_{l\to\infty}v_l=v_0$ and $v_l\in V_{\mathcal{C}_l}$. Hence
\begin{equation*}
v_0=\lim_{l\to\infty} v_l=\lim_{l\to\infty} P_{V_{\mathcal{C}_l}}v_l=P_{V_0}v_0.
\end{equation*}
Hence $V_0=V_{\mathcal{C}_0}$ and we have proven the claim. Passing again to a subsequence, we assume that $(W_l)_{l=1}^{\infty}$ converges to $W_0\in G_{m,k}$. As in Lemma \ref{lem:alpha} we show that there exists $T_0\in O(V_{\mathcal{C}_0},W_0)$ such that
\begin{equation*}
T_kP_{V_{\mathcal{C}_l}}\text{ converges to }T_0P_{V_{\mathcal{C}_0}}.
\end{equation*}

Choose now any $y_0\in (x_0+\mathcal{C}_0)\cap B(x_0,\beta_k)$. Then 
\begin{equation*}
\frac{y_0-x_0}{\norm{y_0-x_0}}\in K_0.
\end{equation*}
Hence, there exists a sequence $(z_l)_{l=1}^{\infty}$ of elements in $K_l$ such that 
\begin{equation*}
\lim_{l\to\infty}z_l=\frac{y_0-x_0}{\norm{y_0-x_0}}.
\end{equation*}
Set
\begin{equation*}
y_l=x_l+\norm{y_0-x_0}z_l.
\end{equation*}
Thus 
\begin{equation*}
\lim_{l\to\infty}y_l=y_0.
\end{equation*}
For sufficiently large $l$,
\begin{equation*}
y_l\in(x_l+ \mathcal{C}_l)\cap B\big(x_l,\big(1-1/l\big)\beta_k(x_l)\big).
\end{equation*}
For $l$ as above, we have
\begin{equation*}
u(x_l)-u(y_l)=T_l(x_l-y_l).
\end{equation*}
Passing to the limit, it follows that
\begin{equation*}
u(x_0)-u(y_0)=T_0(x_0-y_0).
\end{equation*}
That is, $\beta_k(x_0)\geq\beta_k$. The proof is complete.
\end{proof}

\begin{lemma}\label{lem:beta}
A point $x\in\mathbb{R}^n$ belongs to a leaf $\mathcal{S}$ of $u$ of dimension at least $k$ if and only if $\beta_k(x)>0$. A point $x\in\mathbb{R}^n$ belongs to a leaf $\mathcal{S}$ of $u$ of dimension exactly $k$ if and only if $\beta_k(x)>0$ and $\beta_{k+1}(x)=0$. 
\end{lemma}
\begin{proof}
Suppose that $x_0\in\mathbb{R}^n$ belongs to a leaf $\mathcal{S}$ of $u$ of dimension $l\in\{k,\dotsc,m\}$. Let $V$ denote the tangent space of $\mathcal{S}$. Choose a point $x_1\in\mathrm{int}\mathcal{S}$ and $\epsilon_0>0$ so that $B(x_1,\epsilon_0)\cap V\subset\mathcal{S}$. For $\epsilon\in (0,\epsilon_0)$ let 
\begin{equation*}
\mathcal{C}=\big\{x\in\mathbb{R}^n\big|x=\lambda (x_2-x_0)\text{ for some }\lambda\geq 0, x_2\in B(x_1,\epsilon)\cap V\big\}.
\end{equation*}
Then $\mathcal{C}$ is a convex cone containing $0$, of dimension $l$ and such that 
\begin{equation*}
\lambda_l(\mathcal{C}\cap B^n)\geq\epsilon^l,
\end{equation*}
provided that $\epsilon$ is sufficiently small.
Moreover, by convexity of $\mathcal{S}$, $u$ is isometric on $(x_0+\mathcal{C})\cap B(x_0,\epsilon)$, if $\epsilon>0$ is sufficiently small. Hence $\beta_l(x_0)\geq \epsilon>0$.
Conversely, suppose that $\beta_k(x_0)>0$. Then there exist
\begin{equation*}
\epsilon>0\text{, a cone }\mathcal{C}\in C_{n,k}(\epsilon), \text{ a subspace }W\in G_{m,k}\text{, an isometry }T\in O(V_{\mathcal{C}},W)
\end{equation*}
such that
\begin{equation*}
u(x_0)-u(y)=T(x-y)\text{ for all }y\in (x_0+\mathcal{C})\cap B(x_0,\epsilon).
\end{equation*}
With use of the Kuratowski-Zorn lemma choose a leaf $\mathcal{S}$ of~$u$ containing 
\begin{equation*}
(x_0+\mathcal{C})\cap B(x_0,\epsilon).
\end{equation*}
Then the dimension of $\mathcal{S}$ is at least $k$. The second assertion is a trivial consequence of the first assertion.
\end{proof}

\begin{lemma}\label{lem:interior}
A point $x\in\mathbb{R}^n$ belongs to relative interior of a leaf $\mathcal{S}$ of $u$ of dimension $k$ if and only if $\alpha_k(x)>0$ and $\beta_{k+1}(x)=0$.
\end{lemma}
\begin{proof}
Suppose that $x_0$ belongs to the relative interior of a leaf $\mathcal{S}$ of~$u$ of dimension $k$. By the previous lemma $\beta_k(x_0)>0$ and $\beta_{k+1}(x_0)=0$. 
Let $V$ denote the tangent space of $\mathcal{S}$. Then, as $x_0$ is in the relative interior, there exist $\epsilon>0$, $W\in G_{m,k}$ and $T\in O(V,W)$ such that 
\begin{equation*}
u(x_0)-u(y)=T(x_0-y)\text{ for all }y\in B(x_0,\epsilon).
\end{equation*} 
That is $\alpha_k(x_0)\geq \epsilon>0$. 

Conversely, suppose that $\alpha_k(x_0)>0$ and $\beta_{k+1}(x_0)=0$. 
Then there exist $V\in G_{n,k}, W\in G_{m,k}$ and $T\in O(V,W)$ such that
\begin{equation*}
u(x_0)-u(y)=T(x_0-y)\text{ for all }y\in B(x_0,\epsilon)\cap V.
\end{equation*} 
It follows from the Kuratowski-Zorn lemma that $x_0$ belongs to a leaf $\mathcal{S}$ of~$u$. As $\beta_{k+1}(x_0)=0$, this leaf is of dimension $k$ and $x_0$ belongs to the relative interior of $\mathcal{S}$.
\end{proof}

\begin{corollary}\label{col:borel}
Let $k\in\{0,\dotsc,m\}$. Then the union of all leaves of $u$ of dimension $k$ is Borel measurable. Moreover, the union of all relative interiors of leaves of $u$ of dimension $k$ is a Borel set and so is the union of all relative boundaries of leaves of $u$ of dimension $k$.
\end{corollary}

Below we adapt a convention that $\inf\emptyset=\infty$.

\begin{definition}\label{def:gamma}
Let $k\in \{0,\dotsc,m\}$. For $\rho>0$, define $\gamma_{k,\mu}\colon \mathbb{R}^n\times\mathbb{R}^m\to \mathbb{R}\cup\{\infty\}$ by the formula
\begin{equation*}
\gamma_{k,\rho}(x,y)=\inf\Big\{t>0\big| y\in t\big(u(\mathcal{S}_x)-u(x)\big)\text{ and }\norm{y}\leq t\rho\Big\}
\end{equation*}
for $x\in\mathbb{R}^n$ such that $\alpha_k(x)>0$ and $\beta_{k+1}(x)=0$
and
\begin{equation*}
\gamma_{k,\rho}(x,y)=\infty\text{ otherwise.}
\end{equation*}
Here $\mathcal{S}_x$ is the unique leaf of $u$ such that $x\in\mathcal{S}_x$.
\end{definition}

\begin{lemma}
For any $k\in \{0,\dotsc,m\}$ and $\rho>0$, the function $\gamma_{k,\rho}$ is Borel measurable.
\end{lemma}
\begin{proof}
As $\alpha_k$ and $\beta_{k+1}$ are Borel measurable, it is enough to show that $\gamma_{k,\rho}$ is Borel measurable on 
\begin{equation*}
A_k=\big\{(x,y)\in\mathbb{R}^n\times\mathbb{R}^m\big|\alpha_k(x)>0\text{ and }\beta_{k+1}(x)=0\big\}.
\end{equation*}
We claim that $\gamma_{k,\rho}$ is lower-semicontinuous on $A_k$.

Indeed let $(x_l,y_l)_{l=1}^{\infty}$ be a sequence in $A_k$ such that there exists $(x_0,y_0)\in A_k$ and
\begin{equation*}
x_0=\lim_{l\to\infty} x_l \text{ and }y_0=\lim_{l\to\infty}y_l \text{ and such that there exists }\lim_{l\to\infty}\gamma_{k,\rho}(x_l,y_l)=\gamma_k.
\end{equation*}
We shall show that
\begin{equation*}
\gamma_{k,\rho}(x_0,v_0)\leq\gamma_k.
\end{equation*}
We know that there exists sequence $(z_l)_{l=1}^{\infty}$ in $\mathbb{R}^n$ and a sequence $(t_l)_{l=1}^{\infty}$ in $\mathbb{R}$ such that 
\begin{equation}\label{eqn:equality3}
y_l=t_l\big(u(z_l)-u(x_l)\big)\text{, where }z_l\in \mathcal{S}_{x_l}\text{ and }0<t_l<\gamma_{k,\rho}(x_k,y_l)+1/l.
\end{equation}
Moreover, as
\begin{equation*}
\norm{z_l-x_l}=\norm{u(z_l)-u(x_l)}=\norm{y_l-u(x_l)}\leq t_l \rho+\norm{x_l}
\end{equation*}
passing possibly to a subsequence, we may assume that $(z_l)_{l=1}^{\infty}$ converges to some $z_0\in\mathcal{S}_{x_0}$. 
Again passing to a subsequence, we may assume that $(t_l)_{l=1}^{\infty}$ converges to some $t_0\geq 0$. Taking limits in (\ref{eqn:equality3}) we see that
\begin{equation*}
y_0=t_0\big(u(z_0)-u(x_0)\big)\text{ with }z_0\in \mathcal{S}_{x_0}\text{ and }0\leq t_0\leq \gamma_k.
\end{equation*}
Hence
\begin{equation*}
y_0\in t_0 \big(u(\mathcal{S}_{x_0})-u(x_0)\big)\text{ and }\norm{y_0}\leq t_0\rho.
\end{equation*}
It follows that 
\begin{equation*}
\gamma_{k,\rho}(x_0,y_0)\leq t_0\leq\gamma_k.
\end{equation*}
The proof is complete.
\end{proof}

\begin{definition}
For a convex set $K\subset\mathbb{R}^m$, such that $0\in \mathrm{int}K$, define \emph{Minkowski functional} of $K$
\begin{equation*}
\norm{\cdot}_K\colon\mathbb{R}^m\to\mathbb{R}\cup\{\infty\}
\end{equation*}
by the formula
\begin{equation*}
\norm{y}_K=\inf\big\{t>0|y\in tK\big\}.
\end{equation*}
\end{definition}

\begin{proposition}\label{pro:minkowski}
Let $K\subset \mathbb{R}^m$ be a convex set such that $0\in\mathrm{int}K$. A point $y\in\mathbb{R}^m$ belongs to the relative interior of $K$ if and only if $\norm{y}_K<1$.

Moreover, if $K$ is compact, then a point $y\in \mathbb{R}^m$ belongs to the boundary of $K$ if and only if $\norm{y}_K=1$.
\end{proposition}
\begin{proof}
If $y\in \mathrm{int}K$, then, as $0+y=y\in\mathrm{int}K$, it follows by continuity of addition, that $y+w\subset\mathrm{int}K$ provided that $\norm{w}\leq\epsilon$, for $\epsilon>0$ sufficiently small. Observe that $\norm{y/s}\leq \epsilon$ if $s\geq \norm{y}/\epsilon$ and thus for large $s>0$
\begin{equation*}
(1+1/s)y\in K.
\end{equation*}
Hence $\norm{y}_K\leq \frac{s}{s+1}<1$.

Conversely, suppose that $\norm{y}_K<1$. Then $y\in tK$ for some $t<1$. As $0\in\mathrm{int}K$, there exists $\epsilon>0$ such that if $\norm{w}\leq \epsilon$, then $w\in K$. Hence, if $\norm{w}\leq\epsilon (1-t)$, then
\begin{equation*}
y+w\in tK+(1-t)K=K,
\end{equation*}
by convexity of $K$.

Assume that $K$ is compact. Suppose that $y\in\partial K$. Then clearly $\norm{y}_K\leq 1$ and, by the above $\norm{y}_K\geq 1$.

Conversely, let $\norm{y}_K=1$. Then there exists a sequence of positive numbers $(t_l)_{l=1}^{\infty}$ converging to $0$ and a sequence $(x_l)_{l=1}^{\infty}$ in $K$ such that
\begin{equation*}
y=(1+t_l)x_l.
\end{equation*}
Taking a convergent subsequence from $(x_l)_{l=1}^{\infty}$ we see that $y=x_0$ for some $x\in K$.
\end{proof}

\begin{lemma}\label{lem:minkowski}
If $x\in\mathbb{R}^n$ belongs to relative interior of a leaf $\mathcal{S}$ of $u$ of dimension at $k$, then $\gamma_{k,\rho}(x,\cdot)$ is Minkowski functional a closed, convex set
\begin{equation*}
K_{\rho}=\big(u(\mathcal{S})-u(x)\big)\cap\big\{y\in\mathbb{R}^m\big|\norm{y}\leq\rho\big\}\subset\mathbb{R}^m.
\end{equation*}
If $x\in\mathbb{R}^n$ does not belong to relative interior of any leaf of dimension $k$, then 
\begin{equation*}
\gamma_{k,\rho}(x,\cdot)=\infty.
\end{equation*}
\end{lemma}
\begin{proof}
Suppose that $x\in\mathbb{R}^n$ does not belong to relative interior of a leaf of $u$ of dimension at least $k$. Then Lemma \ref{lem:interior} and Definition \ref{def:gamma} tells us that $\gamma_{k,\rho}(x)=\infty$. 

Let now $x\in\mathrm{int}\mathcal{S}$, where $\mathcal{S}$ is a $k$-dimensional leaf. By Lemma \ref{lem:boundary}, $x$ belongs to a unique leaf. The assertion of the lemma follows readily from definitions.
\end{proof}

\begin{definition}
Let $k\in \{0,\dotsc,m\}$. We shall denote by $T_k$ union of all $k$-dimensional leaves of $u$, by $\mathrm{int}T_k$ union of all relative interiors of all $k$-dimensional leaves of $u$ and by $\partial{T}_k$ union of all relative boundaries of all $k$-dimensional leaves of $u$.
\end{definition}

\begin{lemma}\label{lem:measurable}
For each $p\in Q$ and each $i,j\in\mathbb{N}$ the cluster $\mathrm{int}T_{pij}$ and its image $G(\mathrm{int}T_{pij})$ are Borel sets. Moreover $\partial T_m$ is a Borel set of Lebesgue measure zero.
\end{lemma}
\begin{proof}
Fix $p\in Q$ and $i,j\in\mathbb{N}$. Recall the Borel set $S_p^i\subset\mathbb{R}^n$ and Lipschitz mapping $w\colon \mathbb{R}^n\to\mathbb{R}^{n-m}$ from Lemma \ref{lem:cover}. Since $w$ is injective on $S_p^i$ it follows from \cite[\S 2.2.10]{Federer} that $w(S_p^i)$ is a Borel subset of $\mathbb{R}^{n-m}$. Moreover, the set $\Lambda$, defined in (\ref{eqn:lambda}), is given by
\begin{equation}\label{eqn:lam}
\Lambda=\Big\{a\in w(S_p^i)\big| \alpha_m(w^{-1}(a))>1/j\Big\}
\end{equation}
as follows by the definition (\ref{eqn:gie}) and Lemma \ref{lem:cover}.
Let $\rho>0$. Definition of the cluster $T_{pij}^{0,\rho}$ implies that
\begin{equation*}
G(T_{pij}^{0,\rho})=\Big\{(a,b)\in\mathbb{R}^{n-m}\times\mathbb{R}^m\big| a\in \Lambda, b\in u(\mathrm{int}\mathcal{S}_{v(a)})-u(v(a)) ,\norm{b}\leq\rho\Big\}.
\end{equation*}
Here $\mathcal{S}_{v(a)}$ is the unique $m$-dimensional leaf of $u$ containing $v(a)$.
Note that Proposition \ref{pro:minkowski} and Lemma \ref{lem:minkowski} tells us that if $a\in\Lambda$, then 
\begin{equation*}
b\text{ belongs to interior of }u(\mathcal{S}_{v(a)})-u(v(a))\cap\big\{y\in\mathbb{R}^m\big|\norm{y}\leq\rho\big\}
\end{equation*}
if and only if
\begin{equation*}
\gamma_{m,\rho}(v(a),b)<1.
\end{equation*}
This is to say,
\begin{equation}\label{eqn:image}
G(T_{pij}^{0,\rho})=\Big\{(a,b)\in\mathbb{R}^{n-m}\times\mathbb{R}^m\big| a\in \Lambda,\gamma_{m,\rho}(v(a),b)<1\Big\}.
\end{equation}
As $\gamma_{m,\rho}$ is Borel measurable, it follows that $G(T_{pij}^{0,\rho})$ is a Borel set.
As
\begin{equation}\label{eqn:mus}
\mathrm{int}T_{pij}=\bigcup_{\rho\in\mathbb{N}} T_{pij}^{0,\rho}
\end{equation}
we conclude that $G(\mathrm{int}T_{pij})$ is Borel as well.

Clearly, $\Lambda$ is also a Borel set. Lemma \ref{lem:efge} shows that $F$, the inverse of $G$ on its image, is well-defined and injective on $G(\mathrm{int}T_{pij})$. 
On the sets $G(T_{pij}^{0,\rho})$, $\rho\in\mathbb{N}$, function $F$ is Lipschitz and 
\begin{equation*}
T_{pij}^{0,\rho}=F(G(T_{pij}^{0,\rho})).
\end{equation*}
Using \cite[\S 2.2.10]{Federer}, we see that $T_{pij}^{0,\rho}$ is a Borel set. Using (\ref{eqn:mus}) again, we see that $\mathrm{int}T_{pij}$ is a Borel set.

We shall show that $\partial T_m$ has Lebesque measure zero. Recall, that Corollary \ref{col:borel} tells us that $\partial T_m$ is a Borel set. 
Consider the set
\begin{equation*}
\mathcal{G}_{\rho}=\Big\{(a,b)\in\mathbb{R}^{n-m}\times\mathbb{R}^m\big| a\in\mathrm{cl}\Lambda,\gamma_{m,\rho}(v(a),b)=1\Big\}.
\end{equation*}
By Fubini's theorem, $\lambda(\mathcal{G}_{\rho})=0$, as boundaries of convex sets have Lebesgue measure zero. 

Recall that $F$ is a Lipschitz map on $G(T_{pij}^{0,\rho})$. Using the Kirszbraun theorem (see e.g \cite{Kirszbraun, Schoenberg}) we extend the restriction of $F$ to $G(T_{pij}^{0,\rho})$ to a Lipschitz map $F_{\rho}$ on $\mathbb{R}^{n-m}\times\mathbb{R}^m$.

Now, for any such extension,
\begin{equation*}
F_{\rho}(\mathcal{G}_{\rho})\supset\partial T_m\cap \big\{x\in T_{pij}\big| \norm{u(x)-p}\leq\rho\big\}.
\end{equation*}
Indeed, let 
\begin{equation*}
x\in \partial T_m\cap\big\{x\in T_{pij}\big| \norm{u(x)-p}\leq\rho\big\}.
\end{equation*}
Choose a sequence $(x_l)_{l=1}^{\infty}$ in $T_{pij}^{0,\rho}$ converging to $x$. The sequence $(G(x_l))_{l=1}^{\infty}$ is bounded by (\ref{eqn:image}) and by (\ref{eqn:lam}). Hence, passing to a subsequence we may assume that it converges to some 
\begin{equation*}
(a,b)\in\mathbb{R}^{n-m}\times\mathbb{R}^m.
\end{equation*}
If $(a,b)\in G(T_{pij}^{0,\rho})$, then there would exist $x'\in T_{pij}^{0,\rho}$ with $G(x')=(a,b)$ and thus
\begin{equation*}
x'=F_{\rho}(a,b)=\lim_{l\to\infty}F(G(x_l))=\lim_{n\to\infty}x_l=x.
\end{equation*}
This would contradict the fact that $x\in \partial T_m$. Hence $(a,b)\notin G(T_{pij}^{0,\rho})$. It follows that $(a,b)$ belongs to the boundary of $G(T_{pij}^{0,\rho})$, which is contained in $\mathcal{G}_{\rho}$.

Therefore we can use $\lambda(\mathcal{G}_{\rho})=0$ and the fact that images under Lipschitz maps of sets of Lebesgue measure zero have Lebesgue measure zero (see e.g. \cite[\S 3.2.3]{Federer}), to conclude that 
\begin{equation*}
\lambda\big(\partial T_m \cap \big\{x\in T_{pij}\big| \norm{u(x)-p}\leq\rho\big\}\big)=0,
\end{equation*}
and hence is Lebesgue measurable. By Lemma \ref{lem:cluster} the sets $T_{pij}$ form a countable covering of $\partial T_m$. It follows that $\lambda(\partial T_m)=0$. This concludes the proof.
\end{proof}

\begin{corollary}
For any $p\in Q$, $i,j\in\mathbb{N}$, the set $T_{pij}$ is Lebesgue measurable.
\end{corollary}
\begin{proof}
$T_{pij}$ is a union of a Borel set $\mathrm{int}T_{pij}$ and a set $\partial T_m\cap T_{pij}$ of Lebesgue measure zero.
\end{proof}

\begin{remark}
The clusters $T_{pij}$ may be taken to be disjoint. Indeed, let $(T_k)_{k=1}^{\infty}$ be a renumbering of the set of clusters. Set for $l\in\mathbb{N}$
\begin{equation*}
T_l'=T_l\setminus \bigcup_{n=1}^{l-1}T_n
\end{equation*} 
and 
\begin{equation*}
\mathrm{int}T_l'=\mathrm{int}T_l\setminus \bigcup_{n=1}^{l-1}\mathrm{int}T_n.
\end{equation*} 
Note that the structure of the clusters $T'_{pij}$ remains the same. For each $T_{pij}$ there exists a Borel subset $S_{pij}=T_{pij}\cap S_p^i$ of $S_p^i\subset\mathbb{R}^n$ on which there are Lipschitz maps 
\begin{equation*}
w\colon\mathbb{R}^n\to\mathbb{R}^{n-m}\text{ and }v\colon\mathbb{R}^{n-m}\to\mathbb{R}^n\end{equation*}
such that
\begin{equation*}
v(w(x))=x\text{ for all } x\in S_{pij}.
\end{equation*}
Indeed, the new cluster is a subset of the old one, so the former maps suffice. From the modification procedure it follows also that Lemma \ref{lem:cluster} still holds true. Moreover, the leaf $\mathcal{S}$ corresponding to a point $z\in S_p\cap S_{pij}$ satisfies
\begin{equation*}
\mathrm{dist}(z,\partial\mathcal{S})>1/j.
\end{equation*}
Also the assertions of Lemma \ref{lem:efge} hold true with the old maps and so does the assertions of Lemma \ref{lem:measurable}, as follows from the modification procedure.
\end{remark}

\section{Disintegration of measure}\label{sec:disin}

The aim of this section is to prove the following theorem.

\begin{theorem}\label{thm:dis}
Let $u\colon\mathbb{R}^n\to\mathbb{R}^m$ be a $1$-Lipschitz map with respect to the Euclidean norms. Then there exists a map $\mathcal{S}\colon\mathbb{R}^n\to CC(\mathbb{R}^n)$ such that for $\lambda$-almost every $x\in\mathbb{R}^n$ the set $\mathcal{S}(x)$ is a maximal closed convex set in $\mathbb{R}^n$ such that $u|_{\mathcal{S}(x)}$ is an isometry. Moreover, there exist a Borel measure on $CC(\mathbb{R}^n)$ and Borel measures $\lambda_{\mathcal{S}}$ such that 
\begin{equation*}
\mathcal{S}\mapsto \lambda_{\mathcal{S}}(A)\text{ is }\nu\text{-measurable for any Borel set }A\subset\mathbb{R}^n
\end{equation*}
and for $\nu$-almost every $\mathcal{S}$ we have $\lambda_{\mathcal{S}}(\mathcal{S}^c)=0$,
and for any $A\subset\mathbb{R}^n$
\begin{equation*}
\lambda(A)=\int_{CC(\mathbb{R}^m)} \lambda_{\mathcal{S}}(A)d\nu(\mathcal{S}).
\end{equation*}
Moreover, for $\nu$-almost every leaf $\mathcal{S}$ of dimension $m$, the measure $\lambda_{\mathcal{S}}$ is equivalent to the restriction to $\mathcal{S}$ of the $m$-dimensional Hausdorff measure.
\end{theorem}

Before the we provide a proof let us define necessary tools and note its several properties.

Let $CL(\mathbb{R}^m)$ denote the space of closed non-empty sets in $\mathbb{R}^m$. On $CL(\mathbb{R}^m)$ we introduce the Wijsman topology (see \cite{Wijsman}). It is the weakest topology such that the mappings
\begin{equation*}
A\mapsto \mathrm{dist}(x,A)
\end{equation*}
are continuous for all $x\in\mathbb{R}^m$. By a result of Beer (see \cite{Beer2}), the set $CL(\mathbb{R}^m)$ equipped with this topology is a Polish space. Let $CC(\mathbb{R}^m)$ denote the set of all closed convex, non-empty sets in $\mathbb{R}^m$. Then $CC(\mathbb{R}^m)$ is a closed subset of $CL(\mathbb{R}^m)$, hence also a Polish space.
Let $X$ be a measurable space. In \cite{Hess} (see also \cite{Beer3}) it is proved that a function $f\colon X\to CL(\mathbb{R}^m)$ is measurable if and only if it is measurable as a multifunction. The latter is defined by the condition that for any open set $U\subset\mathbb{R}^m$ the set
\begin{equation*}
\{x\in X| f(x)\cap U\neq \emptyset\}
\end{equation*}
is measurable in $X$.

Let $X,Y$ be two Polish spaces. Let $\eta$ be a non-negative Borel probability measure on $X$, $T\colon X\to Y$ be a Borel measurable map and let $\nu$ be the push-forward of $\eta$ by $T$, that is a Borel probability measure on $Y$ such that for a Borel set $A$ in $Y$ we have
\begin{equation*}
\nu(A)=\eta(T^{-1}(A)).
\end{equation*}
A \emph{disintegration} of $\eta$ with respect to $T$ is a collection of Borel probability measures $\{\eta_y| y\in Y\}$ on $X$, such that if $y\in T(X)$, then $\eta_y(T^{-1}(y))=1$ for $\nu$-almost every $y\in Y$, if $f$ is an integrable function with respect to $\eta$, then for $\nu$-almost every $y\in Y$, $f$ is integrable with respect to $\eta_y$, the function
\begin{equation*}
y\mapsto \int_X fd\eta_y
\end{equation*}
is $\nu$-measurable, and moreover
\begin{equation*}
\int_X fd\eta=\int_Y\int_X fd\eta_y d\nu.
\end{equation*}
We shall also say that $\{\eta_y|y\in Y\}$ are conditional measures.

We shall use the following theorem (see e.g. \cite{Garling}). We refer also to \cite{Rokhlin} for a more general approach.

\begin{theorem}\label{thm:disinteg}
Suppose that $X,Y$ are Polish spaces and $\eta$ is a Borel probability measure on $X$ and $T\colon X\to Y$ is a Borel map. Then a disintegration of $\eta$ with respect to $T$ exists and moreover it is essentially unique, that is if $\{\eta_y|y\in Y\}$ and $\{\eta_y'|y\in Y\}$ are two disintegrations of $\eta$ then $\eta_y=\eta_y'$ for $\nu$-almost every $y\in Y$.
\end{theorem}

\begin{proof}[Proof of Theorem \ref{thm:dis}]
In the previous sections we have defined leaves $\mathcal{S}$ of $u$. We have proved that for almost every $x\in\mathbb{R}^n$ there is a unique leaf $\mathcal{S}$ that contains $x$ and that the set of non-uniqueness $B(u)$ is contained in a Borel set $N(u)$ of non-differentiability of $u$, which is of measure zero, see Corollary \ref{col:unique}.

We have a well-defined map $\mathcal{S}\colon \mathbb{R}^n\to CC(\mathbb{R}^n)$ that assigns to any $x\in \mathbb{R}^n\setminus N(u)$ a unique leaf $\mathcal{S}(x)$ that contains $x$ and on $N(u)$ we set $\mathcal{S}(x)=\{x\}$.

Note that for any compact set $K\subset\mathbb{R}^n$ the set $\{x\in \mathbb{R}^n| \mathcal{S}(x)\cap K\neq \emptyset\}$ is equal to
\begin{equation*}
\bigcup_{k=0}^m \{x\in \mathbb{R}^n\setminus N(u)| \beta_k(x)>0, \sup\Big\{\frac{\norm{u(x)-u(y)}}{\norm{x-y}\}}=1\big| y\in K\Big\}\cup\big( K\cap N(u)\big).
\end{equation*}
Therefore by, Lemma \ref{lem:beta}, and the fact that the map 
\begin{equation*}
x\mapsto  \sup\Big\{\frac{\norm{u(x)-u(y)}}{\norm{x-y}\}}=1\big| y\in U\Big\}
\end{equation*}
is lower-semicontinuous, and that any open set $U\subset\mathbb{R}^n$ is a countable union of compact sets, the map
 $\mathcal{S}$ is Borel measurable.

We shall use this to obtain the disintegration of measures. Recall that $CC(\mathbb{R}^n)$ and $\mathbb{R}^m$ are Polish spaces and that $\mathcal{S}$ is a Borel measurable map.

Let us now consider a Borel probability measure $\lambda_r$ which is the normalised restriction of the Lebesgue measure to a Borel set $R$ of finite positive Lebesgue measure. Applying the  Theorem \ref{thm:disinteg} to the spaces $\mathbb{R}^n$ and $CC(\mathbb{R}^n)$ and map $\mathcal{S}$ we obtain a disintegration $\{\lambda_{\mathcal{S}}|\mathcal{S}\in CC(\mathbb{R}^n\}$ such that for $\nu$-almost every leaf $\mathcal{S}$ of $u$ we have 
\begin{equation*}
\lambda_{\mathcal{S}}(\mathcal{S})=1,
\end{equation*}
i.e. $\lambda_{\mathcal{S}}$ is concentrated on $\mathcal{S}$, as the preimages of every leaf $\mathcal{S}\in CC(\mathbb{R}^n)$ are exactly sets $\mathcal{S}\subset\mathbb{R}^n$, and
for any set $A\subset\mathbb{R}^n$ the function
\begin{equation*}
\mathcal{S}\mapsto \lambda_{\mathcal{S}}(A)
\end{equation*} 
is $\nu$-measurable and 
\begin{equation*}
\lambda_{r}(A)=\int_{CC(\mathbb{R}^n)}  \lambda_{\mathcal{S}}(A)d\nu(\mathcal{S}).
\end{equation*}
If we let $R$ vary and take a countable partition of $\mathbb{R}^n$ into pairwise disjoint sets of finite and positive Lebesgue measure, then adding up the above conditional measures, we obtain the conditional measures for the full Lebesgue measure.

We shall use the notation from previous sections. 
Fix $p\in Q$ and $i,j\in\mathbb{N}$ and consider the cluster $\mathrm{int}T_{pij}$. Let 
\begin{equation*}
\lambda_{pij}=\lambda|_{\mathrm{int}T_{pij}}.
\end{equation*}
By Lemma \ref{lem:efge}, the map $F$ is a bijection of $G(\mathrm{int}T_{pij})$ and $\mathrm{int}T_{pij}$. As for any $\rho>0$, $F$ is Lipschitz on $T_{pij}^{0,\rho}$ and these sets are a covering of the cluster $\mathrm{int}T_{pij}$ we may apply the area formula (see e.g. \cite[\S 3.2.5]{Federer}) to infer that for any integrable $\phi\colon\mathbb{R}^n\to\mathbb{R}$ 
\begin{equation}\label{eqn:area}
\int_{G(\mathrm{int}T_{pij})}\phi (F(x))J_nF(x)d\lambda(x)=\int_{\mathrm{int}T_{pij}}\phi(z)d\lambda(z).
\end{equation}
Here $J_nF$ denotes the $n$-dimensional Jacobian of $F$. Define a function
\begin{equation*}
f\colon\mathbb{R}^{n-m}\times\mathbb{R}^m\to\mathbb{R}
\end{equation*}
by the formula
\begin{equation*}
f(x)=J_nF(x)\text{ if }x\in G(\mathrm{int}T_{pij})\text{ and }f(x)=0\text{ otherwise.}
\end{equation*}
Observe that $f$ is non-negative and Borel measurable, as $G(\mathrm{int}T_{pij})$ is a Borel set by Lemma \ref{lem:efge}. Putting $\phi=\mathbf{1}_{\mathrm{int}T_{pij}}$ in (\ref{eqn:area}) shows that $f$ is integrable. 

By Fubini's theorem, the functions $f(x,\cdot)$ are integrable for almost every point $x\in\mathbb{R}^{n-m}$ and we have
\begin{equation*}
\int_{\mathbb{R}^{n-m}\times\mathbb{R}^m}\phi (F(z))f(z)d\lambda(z)=\int_{\mathbb{R}^{n-m}}\int_{\mathbb{R}^m}\phi (F(a,b))f(a,b)d\lambda(b)d\lambda(a).
\end{equation*}
Observe now that $(a,b)\in G(\mathrm{int}T_{pij})$ if and only if there exists an $m$-dimensional leaf $\mathcal{S}_a\subset T_{pij}$ intersecting $T_{pij}$ at a point $z$ and a point $x\in\mathcal{S}_a$ such that
\begin{equation*}
a=w(z)\text{ and }b=u(x)-u(z).
\end{equation*}
Note that $F$ on $G(\mathrm{int}\mathcal{S}_a)$ is an isometry. Therefore by a linear change of variables
\begin{equation*}
\int_{G(\mathrm{int}\mathcal{S}_a)}\phi(F(a,b))f(a,b)d\lambda(b)=\int_{\mathrm{int}\mathcal{S}_a}\phi f\circ G d\mathcal{H}_m.
\end{equation*}
Here $\mathcal{H}_m$ is the $m$-dimensional Hausdorff measure on $\mathbb{R}^n$. Let
\begin{equation*}
\Lambda=\big\{a\in\mathbb{R}^{n-m}|(a,0)\in G(\mathrm{int}T_{pij})\big\}.
\end{equation*}
Note that the map 
\begin{equation*}
\Lambda\ni a\mapsto \int_{\mathrm{int}\mathcal{S}_a}\phi f\circ G d\mathcal{H}_m
\end{equation*}
is Borel measurable and that for any integrable Borel measurable function $\phi$ we have
\begin{equation*}
\int_{\mathbb{R}^n}\phi d\lambda_{pij}=\int_{\Lambda}\Big(\int_{\mathrm{int}\mathcal{S}_a}\phi f\circ G d\mathcal{H}_m\Big) d\lambda(a)=\int_{\Lambda}\Big(\int_{\mathcal{S}_a}\phi f d\lambda'_{\mathcal{S}_a}\Big) m(a)d\lambda(a),
\end{equation*}
as the boundaries of convex sets have Hausdorff measures of appropriate dimension zero.
Here 
\begin{equation*}
d\lambda'_{\mathcal{S}_a}=\frac{f\circ G \mathbf{1}_{\mathcal{S}_a}d\mathcal{H}_m}{\int_{\mathcal{S}_a}f\circ G \mathbf{1}_{\mathcal{S}_a}d\mathcal{H}_m}
\end{equation*}
and $m(a)=\int_{\mathcal{S}_a}f\circ G \mathbf{1}_{\mathcal{S}_a}d\mathcal{H}_m$. Clearly $\lambda'_{\mathcal{S}_a}$ is equivalent to the Hausdorff measure on $\mathcal{S}_a$.
Define a map $H\colon \Lambda\to CC(\mathbb{R}^m)$
\begin{equation*}
a\mapsto\mathcal{S}_a
\end{equation*}
that sends a point $a\in\Lambda$ 
to the unique leaf
\begin{equation*}
\mathcal{S}_a=\mathrm{cl}F\big(G(\mathrm{int}T_{pij})\cap \{a\}\times\mathbb{R}^m)\big)
\end{equation*}
such that $a=w(z)$ for a point $z\in \mathrm{int}\mathcal{S}_a\cap T_{pij}$. 
Then $H$ is Borel measurable with respect to the Wijsman topology on $CC(\mathbb{R}^m)$.
Indeed, as noted before, the Borel measurability with respect to the Wijsman topology is equivalent to that for any open set $U\subset \mathbb{R}^m$ the set 
\begin{equation*}
\Big\{a\in\Lambda| U\cap \mathrm{cl} F\big(G(\mathrm{int}T_{pij})\cap \{a\}\times\mathbb{R}^m)\big)\neq\emptyset\Big\}
\end{equation*}
is Borel measurable. Let $\pi$ denote the projection on the first coordinate 
\begin{equation*}
\pi\colon \mathbb{R}^{n-m}\times\mathbb{R}^m\to\mathbb{R}^{n-m}.
\end{equation*}
As $U$ is open the above set is equal to
\begin{equation*}
\Big\{a\in\Lambda| \pi^{-1}(a)\cap G(\mathrm{int}T_{pij})\cap F^{-1}(U)\neq\emptyset\Big\},
\end{equation*}
which is Borel measurable, by the measurability of the map $a\mapsto \pi^{-1}(a)$. Moreover, $H$ is an injection.

By the above considerations we see that
\begin{equation*}
\int_{\mathbb{R}^n}\phi d\lambda_{pij}=\int_{\Lambda}\Big(\int_{\mathbb{R}^n}\phi d\lambda'_{\cdot}\Big)(H(a))m(a)d\lambda(a)=\int_{CC(\mathbb{R}^n)} \Big(\int_{\mathbb{R}^n}\phi d\lambda'_{\mathcal{S}}\Big) d \rho(\mathcal{S}),
\end{equation*}
where $\rho$ is the push forward of the measure $m(a)d\lambda(a)$ by the map $H$. Hence $\{\lambda'_{\mathcal{S}}| \mathcal{S}\in H(\Lambda)\}$ constitutes a disintegration of $\lambda_{pij}$ with respect to the map $\mathcal{S}$. Indeed, it follows by taking $\phi$ to be the indicator function of $\mathcal{S}^{-1}(C)$ for $C\subset CC(\mathbb{R}^n)$ that $\rho=\nu$. 

Applying the above result to each cluster separately we infer that for $\nu$-almost every $\mathcal{S}$ the conditional measures $\lambda'_{\mathcal{S}}$ are equivalent to the restriction of the $m$-dimensional Hausdorff measure to $\mathcal{S}$. 

The uniqueness part of Theorem \ref{thm:disinteg} and the fact that $\partial T_m$ has Lebesgue measure zero, see Lemma \ref{lem:measurable}, implies that the conditional measures $\lambda_{\mathcal{S}}$ are $\nu$-almost surely equivalent to the restriction of $\mathcal{H}_m$ to $\mathcal{S}$.  
\end{proof}


\begin{corollary}\label{col:dis}
Let $u\colon\mathbb{R}^n\to\mathbb{R}^m$ be a $1$-Lipschitz map with respect to the Euclidean norms. Let $\mu$ be a Borel measure on $\mathbb{R}^n$ that is absolutely continuous with respect to the Lebesgue measure. Then there exists a map $\mathcal{S}\colon\mathbb{R}^n\to CC(\mathbb{R}^n)$ such that for $\lambda$-almost every $x\in\mathbb{R}^n$ the set $\mathcal{S}(x)$ is a maximal closed convex set in $\mathbb{R}^n$ such that $u|_{\mathcal{S}(x)}$ is an isometry. Moreover, there exist a Borel measure on $CC(\mathbb{R}^n)$ and Borel measures $\mu_{\mathcal{S}}$ such that 
\begin{equation*}
\mathcal{S}\mapsto \mu_{\mathcal{S}}(A)\text{ is }\nu\text{-measurable for any Borel set }A\subset\mathbb{R}^n
\end{equation*}
and for $\nu$-almost every $\mathcal{S}$ we have $\mu_{\mathcal{S}}(\mathcal{S}^c)=0$,
and for any $A\subset\mathbb{R}^n$
\begin{equation*}
\mu(A)=\int_{CC(\mathbb{R}^n)} \mu_{\mathcal{S}}(A)d\nu(\mathcal{S}).
\end{equation*}
Moreover, for $\nu$-almost every leaf $\mathcal{S}$ of dimension $m$, the measure $\mu_{\mathcal{S}}$ is absolutely continuous with respect to the $m$-dimensional Hausdorff measure.
\end{corollary}
\begin{proof}
Follows directly from Theorem \ref{thm:dis}.
\end{proof}
%
%
\section{Optimal transport for vector measures}\label{sec:transport}

In this section we study the following variational problem. Let $\mu$ be a Borel, $\mathbb{R}^m$-valued measure such that $\mu(\mathbb{R}^m)=0$. We consider 
\begin{equation}\label{eqn:sup}
\sup\Big\{\int_{\mathbb{R}^n}\langle u,d\mu\rangle\big| u\colon\mathbb{R}^n\to\mathbb{R}^m\text{ is }1\text{-Lipschitz}\Big\}.
\end{equation}
Suppose that $\mu$ is absolutely continuous with respect to the Lebesgue measure. It was conjectured in \cite{Klartag} that if $u$ attains the above supremum, then the disintegration 
\begin{equation*}
\{\norm{\mu}_{\mathcal{S}}|\mathcal{S}\in CC(\mathbb{R}^m)\}
\end{equation*}
of $\norm{\mu}$ with respect to the partition formed by the leaves of $u$ satisfy
\begin{equation*}
\int_{\mathbb{R}^n} \frac{d\mu}{d\norm{\mu}}d\norm{\mu}_{\mathcal{S}}=0.
\end{equation*}
We provide a counterexample to this conjecture.

We also develop theory of optimal transport for vector measures, which provides a dual problem for (\ref{eqn:sup}).

\begin{definition}\label{defin:vari}
Let $\Omega$ be a topological space and let $\pi\colon \mathcal{B}(\Omega)\to\mathbb{R}^m$ be a vector measure on the $\sigma$-algebra $\mathcal{B}(\Omega)$ of Borel subsets of $\Omega$. We define its \emph{total variation} $\norm{\pi}\colon \mathcal{B}(\Omega)\to\mathbb{R}$ by
\begin{equation}\label{eqn:vari}
\norm{\pi}(A)=\sup\Big\{\sum_{i=1}^{\infty}\norm{\pi(A_i)}\big| A= \bigcup_{i=1}^{\infty} A_i, A_i\in\mathcal{B}(\Omega), A_i\cap A_j=\emptyset, i,j\in\mathbb{N}\Big\}
\end{equation}
for all $A\in\mathcal{B}(\Omega)$.
\end{definition}

It can be shown (see \cite{RudinRC}) that total variation of a vector measure is a non-negative finite measure.

Let $X$ be a metric space with metric $d$. Let $\mu$ be $\mathbb{R}^m$-valued measure on Borel $\sigma$-algebra $\mathcal{B}(X)$ of $X$. If $\pi$ is a $\mathbb{R}^m$-valued measure on Borel $\sigma$-algebra $\mathcal{B}(X\times X)$, we write $\mathrm{P}_1\pi$ for the first \emph{marginal} of $\pi$, i.e. the measure given by 
\begin{equation*}
\mathrm{P}_1\pi(A)=\pi(A\times X),
\end{equation*}
for all $A\in\mathcal{B}(X)$, and $\mathrm{P}_2\pi$ for the second \emph{marginal} of $\pi$, 
\begin{equation*}
\mathrm{P}_2\pi(B)=\pi(X\times B),
\end{equation*}
for all $B\in\mathcal{B}(X)$. We shall consider an optimization problem
\begin{equation}\label{eqn:KR}
\mathcal{I}(\mu)=\inf\bigg\{{\int_{X\times X}}d(x,y) d\norm{\pi}(x,y)\Big| \pi\in \Gamma(\mu) \bigg\}.
\end{equation}
Here $\Gamma(\mu)$ is the set of all $\mathbb{R}^m$-valued measures $\pi$ on $\mathcal{B}(X\times X)$ such that 
\begin{equation*}
\mu=\mathrm{P}_1\pi -\mathrm{P}_2\pi .
\end{equation*}
To check whether (\ref{eqn:KR}) defines a meaningful quantity, we have to check if $\Gamma(\mu)$ is non-empty.

We shall need the following definition.

\begin{definition}\label{defin:prod}
Let $\mathcal{F},\mathcal{G}$ be two $\sigma$-algebras on $X,Y$ respectively. Let $\sigma\colon \mathcal{F}\to\mathbb{R}^m$ and let $\theta \colon \mathcal{G}\to\mathbb{R}$ be two measures. An unique measure $\sigma\otimes \theta \colon \mathcal{F}\otimes\mathcal{G}\to\mathbb{R}^n$ such that
\begin{equation*}
\langle \sigma\otimes\theta,v\rangle= \langle \sigma, v\rangle \otimes \theta
\end{equation*}
for all $v\in\mathbb{R}^m$ we shall call the \emph{product measure}.
Here  $\langle \sigma,v \rangle \otimes \theta$ is the usual product measure of $\mathbb{R}$-valued measures.
\end{definition}

\begin{remark}
It is clear that the product measure exists. The product measure $\theta\otimes \sigma$ for measures $\sigma\colon \mathcal{F}\to\mathbb{R}^m$ and $\theta\colon\mathcal{G}\to\mathbb{R}$ is defined analogously.
\end{remark}
%

\begin{proposition}
$\Gamma(\mu)$ is non-empty if and only if 
\begin{equation}\label{eqn:equam}
\mu(X)=0.
\end{equation}
\end{proposition}
\begin{proof}
Clearly, if there exists $\pi\in\Gamma(\mu)$, then
\begin{equation*}
\mu(X)=\mathrm{P}_1\pi(X)-\mathrm{P}_2\pi(X)=\pi(X\times X)-\pi(X\times X)=0,
\end{equation*}
so the condition (\ref{eqn:equam}) is satisfied. Conversely, assume that (\ref{eqn:equam}) holds true. If $\mu$ is equal to zero, then $\pi=0$ belongs to $\Gamma(\mu)$. Let $\nu$ be any Borel probability measure on $X$. Set
\begin{equation*}
\pi=\mu\otimes \nu.
\end{equation*}
Here $\mu\otimes \nu$ is the product measure, see Definition~\ref{defin:prod}. 
Then for any $A\in\mathcal{B}(X)$, we have 
\begin{equation*}
\pi(A\times X)-\pi(X\times A)=\mu(A).\end{equation*}
This is to say, $\mathrm{P}_1\pi-\mathrm{P}_2\pi=\mu$.
\end{proof}

The quantity defined by (\ref{eqn:KR}) we shall call the Kantorovich-Rubinstein norm of $\mu$ (see e.g. \cite{Villani1, Villani2, Kolesnikov} for references regarding the Monge-Kantorovich problem).

\begin{proposition}
Assume that $\mu(\mathbb{R}^n)=0$. Then $\mathcal{I}(\mu)<\infty$ provided that
\begin{equation}\label{eqn:moment}
\int_{\mathbb{R}^n}d(x,x_0)d\norm{\mu}(x)<\infty
\end{equation}
for some (equivalently: any) $x_0\in X$.
\end{proposition}
\begin{proof}
Define
\begin{equation*}
\pi=\mu\otimes \delta_{x_0}.
\end{equation*}
Here $\delta_{x_0}$ is a probability measure such that $\delta_{x_0}(\{x_0\})=1$.
Then $\pi\in\Gamma(\mu)$ and
\begin{equation}\label{eqn:finite}
\int_{X\times X}d(x,y)d\norm{\pi}(x,y)\leq \int_{X}d(x,x_0)d\norm{\mu}(x).
\end{equation}
This shows that $\mathcal{I}(\mu)<\infty$, provided that (\ref{eqn:moment}) is satisfied.
The equivalence of finiteness of
\begin{equation*}
\int_{\mathbb{R}^n}d(x,y)d\norm{\mu}(x)<\infty
\end{equation*}
for any $y\in X$ follows by triangle inequality. 
\end{proof}

\begin{definition}
We define the \emph{Wasserstein space} $\mathcal{W}_1(X,\mathbb{R}^m)$ of all Borel measures $\mu$ on $X$ with values in $\mathbb{R}^m$ such that 
\begin{equation*}
\mu(X)=0 \text{ and }
\int_{X}d(x,x_0)d\norm{\mu}(x)<\infty
\end{equation*}
for some $x_0\in X$.
We endow it with a norm $\norm{\mu}_{\mathcal{W}_1(X,\mathbb{R}^m)}=\mathcal{I}(\mu)$.
\end{definition}

Before we proceed let us recall some definitions.

\begin{definition}
Let $X$ be a Hausdorff topological space. We say that a non-negative measure $\mu\colon\mathcal{B}(X)\to\mathbb{R}$ is \emph{inner regular} if for any Borel set $B\in\mathcal{B}(X)$ we have
\begin{equation*}
\mu(B)=\sup\{\mu(K)|K\subset B, K \text{ is a compact set}\}.
\end{equation*}
We say that $\mu$ is \emph{locally finite} if for any $x\in X$ there exists a neighbourhood $U$ of $x$ such that
\begin{equation*}
\mu(U)<\infty.
\end{equation*}
We say that $\mu$ is a \emph{Radon measure} if it is inner regular and locally finite.
We say that $X$ is a \emph{Radon space} if every Borel probability measure on $X$ is a Radon measure.  
\end{definition}

\begin{lemma}\label{lem:lip}
Suppose that $X$ is a Radon space. Let $\mu\colon\mathcal{B}(X)\to\mathbb{R}^m$ be a Borel measure. Suppose that for any Lipschitz function $u\colon X\to\mathbb{R}^m$
\begin{equation*}
\int_{X}\langle u,  d\mu\rangle =0.
\end{equation*}
Then $\mu=0$.
\end{lemma}
\begin{proof}
We may assume that $m=1$. Let $\mu=\mu_+-\mu_-$ be the Hahn-Jordan decomposition of $\mu$. There exists two disjoint Borel sets $A,B\subset X$ with $\mu_+(A^c)=0$ and $\mu_-(B^c)=0$. Choose any Borel set $E\subset A$. As any finite measure on $X$ is inner regular, for any $\epsilon>0$, there exists a compact set $K\subset E$ such that
\begin{equation*}
\mu_+(E)\leq \mu_+(K)+\epsilon.
\end{equation*}
Define a function $u_{\epsilon}$ by the formula
\begin{equation*}
u_{\epsilon}(x)=(1-\frac{1}{\epsilon}\mathrm{dist}(x,K))\vee 0.
\end{equation*}
Then $u_{\epsilon}$ is Lipschitz, equal to $1$ on $K$ and equal to $0$ on the complement of \begin{equation*}
K_{\epsilon}=\{x\in X|\mathrm{dist}(x,K)\leq \epsilon\}.
\end{equation*}
Thus
\begin{equation*}
0=\int_{X}u_{\epsilon}d\mu=\mu_+(K)+\int_{K_{\epsilon}\setminus K}u_{\epsilon}d\mu,
\end{equation*}
Therefore, by the above,
\begin{equation*}
\mu_+(E)\leq \epsilon+\mu_+(K)\leq \epsilon+\mu_+(K_{\epsilon}\setminus K).
\end{equation*}
Letting $\epsilon\to 0$, we get $\mu_+(E)=0$. It follows that $\mu_+=0$. By symmetry, $\mu_-=0$. This is to say, $\mu=0$.
\end{proof}

\begin{remark}
In what follows, we shall always assume that underlying space $X$ is a Radon space.
\end{remark}

\begin{proposition}
The function $\mathcal{W}_1(X,\mathbb{R}^m)\ni\mu\mapsto\norm{\mu}_{\mathcal{W}_1(X,\mathbb{R}^m)}\in\mathbb{R}$ is a norm.
\end{proposition}
\begin{proof}
Let us first check that 
\begin{equation}\label{eqn:nondeg}
\norm{\mu}_{\mathcal{W}_1(X,\mathbb{R}^m)}=0 \text{ if and only if } \mu=0.
\end{equation}
If $\mu=0$, then $\pi=0$ belongs to $\Gamma(\mu)$, so $\norm{\mu}_{\mathcal{W}_1(X,\mathbb{R}^m)}=0$. Conversely, assume that $\norm{\mu}_{\mathcal{W}_1(X,\mathbb{R}^m)}=0$.
Choose any $L$-Lipschitz function 
\begin{equation*}
u\colon X\to\mathbb{R}^m.
\end{equation*}
Then for any $\pi\in\Gamma(\mu)$ we have
\begin{equation*}
\Big\lvert\int_{X}\langle u, d\mu\rangle\Big\rvert= \Big\lvert\int_{X\times X}\langle u(x)-u(y), d\pi(x,y)\rangle\Big\rvert\leq L \int_{X\times X}d(x,y)d\norm{\pi}(x,y).
\end{equation*}
Therefore if $\norm{\mu}_{\mathcal{W}_1(X,\mathbb{R}^m)}=0$, then
\begin{equation*}
\int_{X}\langle u, d\mu\rangle=0.
\end{equation*}
It follows by Lemma \ref{lem:lip}, that $\mu=0$.
Homogeneity of  $\norm{\cdot}_{\mathcal{W}_1(X,\mathbb{R}^m)}$ is clear. Let us show that the triangle inequality holds. 
For this choose measures $\mu,\nu\in \mathcal{W}_1(X,\mathbb{R}^m)$ and any measures $\pi\in \Gamma(\mu)$ and $\rho\in\Gamma(\nu)$. Then
\begin{equation*}
\mu+\nu=\mathrm{P}_1(\pi+\rho)-\mathrm{P}_2(\pi+\rho),
\end{equation*}
so that $\pi+\rho\in\Gamma(\mu+\nu)$.
It follows that 
\begin{equation*}
\begin{aligned}
\norm{\mu+\nu}_{\mathcal{W}_1(X,\mathbb{R}^m)}&\leq \int_{X\times X}d(x,y)d\norm{\pi+\rho}(x,y)\leq\\
&\leq \int_{X\times X}d(x,y)d\norm{\pi}(x,y)+ \int_{\mathbb{R}^n\times\mathbb{R}^n}d(x,y)d\norm{\rho}(x,y).
\end{aligned}
\end{equation*}
Taking infimum over all $\pi,\rho$ we see that the triangle inequality holds.
\end{proof}

\begin{proposition}\label{pro:density}
The linear space $\mathcal{F}$ of measures of the form
\begin{equation*}
\sum_{i=1}^n \delta_{x_i} v_i
\end{equation*}
for $x_i\in X$ and $v_i\in\mathbb{R}^m$, $i=1,\dotsc,n$, such that $\sum_{i=1}^n v_i=0$, is dense in $\mathcal{W}_1(X,\mathbb{R}^m)$.
\end{proposition}
\begin{proof}
Choose any measure $\mu\in\mathcal{W}_1(X,\mathbb{R}^m)$. Choose any $\epsilon>0$. Choose any point $x_0\in X$ and a compact set $K$ such that
\begin{equation*}
\int_{K^c}d(x,x_0)d\norm{\mu}(x)\leq \epsilon.
\end{equation*}
Choose pairwise disjoint Borel sets $A_1,A_2,\dotsc,A_k \subset K$ such that the diameter of each is at most $\epsilon$ and
\begin{equation*}
K=\bigcup_{i=1}^kA_i.
\end{equation*}
Consider the restrictions $\mu_i=\mu|_{A_i}$ of the measure $\mu$ to the sets $A_i$, $i=1,2,\dotsc,k$. Choose any points $x_i\in A_i$. Then, as 
\begin{equation*}
\pi_i=\mu_i\otimes\delta_{x_i}\in\Gamma(\mu_i-\mu_i(X)\delta_{x_i}),
\end{equation*}
we have
\begin{equation*}
\norm{\mu_i- \mu_i(X) \delta_{x_i}}_{\mathcal{W}_1(X,\mathbb{R}^m)}\leq \int_{X}d(y,x_i)d\norm{\mu_i}(y)\leq \epsilon \norm{\mu}(A_i).
\end{equation*}
Let $\mu_0=\mu|_{K^c}$ and $A_0=K^c$. Then 
\begin{equation*}
\pi_0=\mu_0\otimes \delta_{x_0}\in\Gamma(\mu_0-\mu_0(X)\delta_{x_0}),
\end{equation*}
so
\begin{equation*}
\norm{\mu_0-\mu_0(X)\delta_{x_0}}_{\mathcal{W}_1(X,\mathbb{R}^m)}\leq \int_{X}d(x,x_0)d\norm{\mu_0}(x)\leq \epsilon.
\end{equation*}
Set 
\begin{equation*}
\nu=\sum_{i=0}^k\mu(A_i)\delta_{x_i}.
\end{equation*}
Then $\nu\in \mathcal{F}$. By triangle inequality
\begin{equation*}
\begin{aligned}
&\norm{\mu-\nu}_{\mathcal{W}_1(X,\mathbb{R}^m)}\leq \sum_{i=0}^k \norm{\mu_i-\mu_i(X)\delta_{x_i}}_{\mathcal{W}_1(X,\mathbb{R}^m)}\leq\\
&\leq \epsilon \sum_{i=1}^k\norm{\mu(A_i)}+\epsilon \leq \epsilon ( \norm{\mu}(X)+1).
\end{aligned}
\end{equation*}
This concludes the proof.
\end{proof}

\begin{corollary}
If $X$ is separable, then so is the Wasserstein space $\mathcal{W}_1(X,\mathbb{R}^m)$.
\end{corollary}
\begin{proof}
Choose a countable dense subset $A\subset X$ and a countable dense set $B\subset\mathbb{R}^n$. Consider a measure $\mu$ given by
\begin{equation*}
\mu=\sum_{i=1}^n \delta_{x_i} v_i
\end{equation*}
for $x_i\in X$ and $v_i\in\mathbb{R}^n$, $i=1,\dotsc,n$, such that $\sum_{i=1}^n v_i=0$. Choose $\epsilon>0$ and $\tilde{x}_i\in A$ and $\tilde{v}_i\in B$, $i=1,\dotsc,n$, such that
\begin{equation*}
d(x_i,\tilde{x}_i)<\epsilon \text{ and } \norm{v_i-\tilde{v}_i}<\epsilon \text{ and } \sum_{i=1}^n\tilde{v}_i=0.
\end{equation*} 
Set 
\begin{equation*}
\tilde{\mu}=\sum_{i=1}^n \delta_{\tilde{x}_i} \tilde{v}_i.
\end{equation*}
Then
\begin{equation*}
\norm{\mu-\tilde{\mu}}_{\mathcal{W}_1(X,\mathbb{R}^m)}\leq \Big\lVert \sum_{i=1}^n \delta_{x_i}(v_i-\tilde{v}_i)\Big\rVert_{\mathcal{W}_1(X,\mathbb{R}^m)}+\Big\lVert \sum_{i=1}^n (\delta_{x_i}-\delta_{\tilde{x}_i})v_i\Big\rVert_{\mathcal{W}_1(X,\mathbb{R}^m)}  
\end{equation*}
Choose any $x_0\in X$. Taking 
\begin{equation*}
\pi=\sum_{i=1}^n \delta_{x_i} \otimes \delta_{x_0}(v_i-\tilde{v}_i)\text{ and }
\rho=\sum_{i=1}^n (\delta_{x_i}\otimes \delta_{\tilde{x}_i})v_i
\end{equation*}
we see that
\begin{equation*}
\Big\lVert \sum_{i=1}^n \delta_{x_i}(v_i-\tilde{v}_i)\Big\rVert_{\mathcal{W}_1(X,\mathbb{R}^m)}\leq \epsilon \sum_{i=1}^nd(x_i,x_0)
\end{equation*}
and
\begin{equation*}
\Big\lVert \sum_{i=1}^n (\delta_{x_i}-\delta_{\tilde{x}_i})v_i\Big\rVert_{\mathcal{W}_1(X,\mathbb{R}^m)}\leq \epsilon \sum_{i=1}^n\norm{v_i}.
\end{equation*}
The conclusion follows now from Proposition \ref{pro:density}.
\end{proof}

\begin{definition}
Choose any $x_0\in X$. Define
\begin{equation*}
\mathcal{L}(X,\mathbb{R}^m)=\{u\colon X\to\mathbb{R}^m|u \text{ is Lipschitz and } u(x_0)=0\},
\end{equation*}
i.e. the Banach space of $\mathbb{R}^m$-valued Lipschitz functions on $\mathbb{R}^n$ taking $0$ value at $x_0$,
with norm 
\begin{equation*}
\norm{u}_{\mathcal{L}(X,\mathbb{R}^m)}=\sup\bigg\{\frac{\norm{u(x)-u(y)}}{d(x,y)}\Big|x,y\in X, x\neq y\bigg\}.
\end{equation*}
\end{definition}

\begin{proposition}\label{pro:reciprocals}
Define 
\begin{equation*}
T\colon\mathcal{L}(X,\mathbb{R}^m)\to\mathcal{W}_1(X,\mathbb{R}^m)^*
\end{equation*}
and 
\begin{equation*}
S\colon\mathcal{W}_1(X,\mathbb{R}^m)^*\to \mathcal{L}(X,\mathbb{R}^m)
\end{equation*}
by 
\begin{equation}\label{eqn:isodual}
T(u)(\mu)=\int_{X}\langle u,d\mu\rangle
\end{equation}
and
\begin{equation}\label{eqn:dualiso}
\langle S(\lambda)(x),w\rangle=\lambda((\delta_x-\delta_{x_0})w),
\end{equation}
for any $w\in\mathbb{R}^m$. Then $S, T$ are mutual reciprocals and establish an isometric isomorphism of $\mathcal{L}(X,\mathbb{R}^m)$ and $\mathcal{W}_1(X,\mathbb{R}^m)^*$. 
\end{proposition}
\begin{proof}
Choose any $\pi\in\Gamma(\mu)$. Then $\mathrm{P}_1\pi-\mathrm{P}_2\pi=\mu$. Thus, if $u$ is a Lipschitz map, then
\begin{equation*}
\bigg\lvert \int_X\langle u, d\mu \rangle \bigg\rvert=\bigg\lvert \int_X\langle u(x)-u(y), d\pi(x,y) \rangle \bigg\rvert\leq \norm{u}_{\mathcal{L}(X,\mathbb{R}^m)}\int_X d(x,y) d\norm{\pi}(x,y).
\end{equation*}
Taking infimum over all $\pi\in\Gamma(\mu)$, we see that 
\begin{equation*}
\bigg\lvert \int_X\langle u, d\mu \rangle \bigg\rvert\leq \norm{u}_{\mathcal{L}(X,\mathbb{R}^m)}\norm{\mu}_{\mathcal{W}_1(X,\mathbb{R}^m)}.
\end{equation*}
The above calculation shows that the formula (\ref{eqn:isodual}) defines a continuous functional of norm at most $\norm{u}_{\mathcal{L}(X,\mathbb{R}^m)}$. If $w\in\mathbb{R}^m$ if of norm $1$ and $x,y\in X$, $x\neq y$, then for 
\begin{equation}\label{eqn:mu}
\mu_{x,y,w}=\frac{\delta_x-\delta_y}{d(x,y)}w
\end{equation}
we have $\norm{\mu_{x,y,w}}_{\mathcal{W}_1(X,\mathbb{R}^m)}\leq 1$ and
\begin{equation*}
\int_{\mathbb{R}^n}\langle u,d\mu_{x,y,w}\rangle=\frac{\langle w, u(x)-u(y)\rangle}{d(x,y)}.
\end{equation*}
Thus
\begin{equation*}
\norm{u}_{\mathcal{L}(X,\mathbb{R}^m)}=\norm{T(u)}.
\end{equation*}
We shall now show that $T\circ S=\mathrm{Id}$. Take any functional $\lambda\in\mathcal{W}_1(X,\mathbb{R}^m)^*$. Set 
\begin{equation*}
\sigma_{x,w}=(\delta_x-\delta_{x_0}) w.
\end{equation*}
Then  $S(\lambda)\colon X\to\mathbb{R}^m$ is defined by the formula
\begin{equation*}
\langle S(\lambda)(x),w\rangle=\lambda(\sigma_{x,w}).
\end{equation*}
It is clear that the above formula defines $S(\lambda)$ uniquely. Then we claim that map $v=S(\lambda)$ is $\norm{\lambda}$-Lipschitz. Indeed
\begin{equation*}
\norm{v(x)-v(y)}=\sup\{\langle v(x)-v(y),w\rangle | w\in\mathbb{R}^m, \norm{w}=1\},
\end{equation*}
and as
\begin{equation*}
\langle v(x)-v(y),w\rangle = \lambda (\sigma_{x,w}-\sigma_{y,w})\leq \norm{\lambda}\norm{\sigma_{x,w}-\sigma_{y,w}}_{\mathcal{W}_1(X,\mathbb{R}^m)}
\end{equation*}
we see that 
\begin{equation*}
\norm{v(x)-v(y)}\leq\norm{\lambda}d(x,y), \text{ since } \norm{\sigma_{x,w}-\sigma_{y,w}}_{\mathcal{W}_1(X,\mathbb{R}^m)}\leq d(x,y).
\end{equation*}
Suppose that $\nu=(\delta_x-\delta_y)z$. We compute
\begin{equation*}
T(v)(\nu)=\int_X\langle v,d\nu\rangle=\int_X\langle v,z\rangle  d(\delta_x-\delta_y)= \lambda (\sigma_{x,z}-\sigma_{y,z})=\lambda(\nu).
\end{equation*}
We see that $T(S(\lambda))$ and $\lambda$ are equal on the set spanned by $(\delta_x-\delta_y)z$, where $x,y\in X$, $z\in\mathbb{R}^m$. By Proposition \ref{pro:density}, we see that $T(S(\lambda))$ and $\lambda$ are equal on $\mathcal{W}_1(X,\mathbb{R}^m)$. 

Let us show also that $S\circ T=\mathrm{Id}$. Choose any $w\in\mathbb{R}^m$ and any map $u\in\mathcal{L}(X,\mathbb{R}^m)$. Then
\begin{equation*}
\langle S(T(u)),w\rangle = T(u)((\delta_x-\delta_{x_0})w)=\int_X\langle u, d(\delta_x-\delta_{x_0})w\rangle=\langle u(x),w\rangle,
\end{equation*}
as $u(x_0)=0$. Therefore $S(T(u))=u$.
\end{proof}

\begin{proposition}\label{pro:lip}
For any $\mu\in \mathcal{W}_1(X,\mathbb{R}^m)$
\begin{equation}\label{eqn:Kantorovich-Rubinsetin}
\sup\bigg\{\int_X\langle u, d\mu \rangle | u\colon X\to\mathbb{R}^m \text{ is } 1\text{-Lipschitz}\bigg \}=\norm{\mu}_{\mathcal{W}_1(X,\mathbb{R}^m)}.
\end{equation}
Moreover, there exists $1$-Lipschitz function $u_0$ such that
\begin{equation}\label{eqn:maxi}
\sup\bigg\{\int_X\langle u, d\mu \rangle | u\colon X\to\mathbb{R}^m \text{ is } 1\text{-Lipschitz}\bigg \}=\int_X\langle u_0, d\mu\rangle.
\end{equation}
\end{proposition}
\begin{proof}
Notice first that the left-hand side of (\ref{eqn:Kantorovich-Rubinsetin}) is clearly at most the right-hand side of (\ref{eqn:Kantorovich-Rubinsetin}). Take any $\mu\in\mathcal{W}_1(X,\mathbb{R}^m)$. Then by the Hahn-Banach theorem there exists a continuous linear functional $\lambda$ of norm $1$ such that 
\begin{equation*}
\lambda(\mu)=\norm{\mu}_{\mathcal{W}_1(X,\mathbb{R}^m)}.
\end{equation*}
By Proposition \ref{pro:reciprocals}, we know that $\lambda$ is of the form
\begin{equation*}
\lambda(\mu)=\int_X\langle u_0,d\mu\rangle
\end{equation*}
for some Lipschitz map $u_0$. The Lipschitz constant of $u_0$ is equal to one, as
\begin{equation*}\norm{u_0}_{\mathcal{L}(X,\mathbb{R}^m)}=\norm{\lambda}=1.
\end{equation*}
This completes the proof.
\end{proof}

\begin{definition}
Any $1$-Lipschitz function $u\colon X\to\mathbb{R}^m$ such that (\ref{eqn:maxi}) holds we shall call an \emph{optimal potential} of measure $\mu$.
\end{definition}

\begin{definition}
A measure $\pi\in \Gamma(\mu)$ such that
\begin{equation*}
\norm{\mu}_{\mathcal{W}_1(X,\mathbb{R}^m)}=\int_{X\times X}d(x,y)d\norm{\pi}(x,y)
\end{equation*}
we shall call an \emph{optimal transport} for $\mu$.
\end{definition}

\begin{theorem}\label{thm:ae}
Let $\mu$ be a Borel measure such that $\mu(X)=0$. Let $u\in\mathcal{L}(X,\mathbb{R}^m) $ be a $1$-Lipschitz map. Let $\pi\in\Gamma(\mu)$. The following conditions are equivalent:
\begin{enumerate}[i)]
\item\label{i:opti}
\begin{equation*}
\int_X \langle u,d\mu\rangle=\int_{X\times X}d(x,y)d\norm{\pi}(x,y)=\norm{\mu}_{\mathcal{W}_1(X,\mathbb{R}^m)} ,
\end{equation*}
\item\label{i:borel}
\begin{equation*}
\int_A \langle u(x)-u(y),d\pi(x,y) \rangle=\int_A d(x,y)d\norm{\pi}(x,y)
\end{equation*}
for any Borel set $A\subset X\times X$,
\item \label{i:eren}
\begin{equation*}
\int_X \langle u,d\mu\rangle=\int_{X\times X} d(x,y)d\norm{\pi}(x,y),
\end{equation*}
\item\label{i:optopt}
$u$ is an optimal potential for $\mu$ and $\pi$ is an optimal transport for $\mu$.
\end{enumerate}
Moreover, if the above conditions hold, then
\begin{equation*}
\norm{u(x)-u(y)}=d(x,y)
\end{equation*}
$\norm{\pi}$-almost everywhere.
\end{theorem}
\begin{proof}
Assume that \ref{i:eren}) holds. Observe that
\begin{equation*}
\int_X\langle u, d\mu\rangle=\int_{X\times X} \langle u(x)-u(y),d\pi(x,y) \rangle.
\end{equation*}
As
\begin{equation*}
\int_X\langle u, d\mu\rangle\leq \norm{\mu}_{\mathcal{W}_1(X,\mathbb{R}^m)}\leq \int_{X\times X}d(x,y)d\norm{\pi}(x,y),
\end{equation*}
then by \ref{i:eren}) we see that in the above inequalities we have equalities.
Suppose that \ref{i:opti}) holds. Clearly
\begin{equation*}
\int_A \langle u(x)-u(y),d\pi(x,y) \rangle\leq\int_A d(x,y)d\norm{\pi}(x,y).
\end{equation*}
If we had strict inequality in \ref{i:borel}) for some Borel set $A\subset X\times X$, then the above computations shows that we would get strict inequality in \ref{i:opti}).
Condition \ref{i:optopt}) is reformulation of \ref{i:opti}).
The last part of the theorem follows readily from \ref{i:borel}).
\end{proof}

\begin{definition}
We say that a Borel set $A\subset\mathbb{R}^n$ is a \emph{transport set} associated with $u$ if it is a Borel set enjoying the following property: if $x\in A\setminus B(u)$ and $y\in\mathbb{R}^n$ is such that 
\begin{equation*}
\norm{u(x)-u(y)}=\norm{x-y},
\end{equation*}
then $y\in A$.
\end{definition}
We say that a measure $\mu\in\mathcal{M}(Z,\mathbb{R}^m)$ is \emph{concentrated} on a subset $X\subset Z$ if there is $\norm{\mu}(Z\setminus X)=0$.

\begin{lemma}\label{lem:concent}
Let $\mu\in \mathcal{W}_1(\mathbb{R}^n,\mathbb{R}^m)$ be concentrated on a set $X\subset\mathbb{R}^n$. Then
\begin{equation*}
\norm{\mu}_{\mathcal{W}_1(\mathbb{R}^n,\mathbb{R}^m)}=\norm{\mu}_{\mathcal{W}_1(X,\mathbb{R}^m)}.
\end{equation*}
\end{lemma}
\begin{proof}
The assertion is that
\begin{equation*}
\sup\Big\{\int_{\mathbb{R}^n}\langle u,d\mu\rangle\big| u\colon\mathbb{R}^n\to\mathbb{R}^m\text{ is }1\text{-Lipschitz}\Big\}
\end{equation*}
is equal to
\begin{equation*}
\sup\Big\{\int_X\langle u,d\mu\rangle\big| u\colon X\to\mathbb{R}^m\text{ is }1\text{-Lipschitz}\Big\}.
\end{equation*}
By the Kirszbraun theorem any $1$-Lipschitz function $u\colon X\to\mathbb{R}^m$ extends to a $1$-Lipschitz function $\tilde{u}\colon\mathbb{R}^n\to\mathbb{R}^m$. Clearly, for any such extension
\begin{equation*}
\int_{\mathbb{R}^n}\langle \tilde{u},d\mu\rangle=\int_X\langle u,d\mu\rangle.
\end{equation*}
The assertion follows.
\end{proof}

Suppose that $\mu\in\mathcal{W}_1(\mathbb{R}^n,\mathbb{R}^m)$ is absolutely continuous with respect to the Lebesgue measure.
The following theorem shows that if there exists an optimal transport for $\mu$ such that its total variation has absolutely continuous marginals, then the conjecture of Klartag holds true. Note that such existence is clear for $m=1$.

\begin{theorem}\label{thm:balance}
Suppose that $\mu\in\mathcal{W}_1(\mathbb{R}^n,\mathbb{R}^m)$ is absolutely continuous with respect to the Lebesgue measure on $\mathbb{R}^n$. Let $u$ be an optimal potential for $\mu$. Then each of the following conditions implies the subsequent one:
\begin{enumerate}[i)]
\item\label{i:continuity}
there exists an optimal transport $\pi$ of $\mu$ such that
\begin{equation}\label{eqn:abs}
\mathrm{P}_1\norm{\pi}\ \text{ is absolutely continuous with respect to }\norm{\mu},
\end{equation}
\item\label{i:desinteg}
for any transport set $A$ associated with $u$: 
\begin{enumerate}[a)]
\item\label{i:transport}
$\pi|_{A\times A}\in \Gamma(\mu|_A)$ is an optimal transport of $\mu|_A$; in particular $\mu(A)=0$,
\item 
$u$ is an optimal potential of $\mu|_A$.
\end{enumerate} 
\item\label{i:dezintegracja}
if $\{\norm{\mu}_{\mathcal{S}}| \mathcal{S}\in CC(\mathbb{R}^n)\}$ is a disintegration of $\norm{\mu}$ with respect to $\mathcal{S}\colon\mathbb{R}^n\to CC(\mathbb{R}^n)$, then for $\nu$-almost every $\mathcal{S}\in CC(\mathbb{R}^n)$ we have
\begin{equation*}
\int_{\mathbb{R}^n}\frac{d\mu}{d\norm{\mu}}d\norm{\mu}_{\mathcal{S}}=0
\end{equation*}
and $u$ is an optimal potential of $\frac{d\mu}{d\norm{\mu}}d\norm{\mu}_{\mathcal{S}}$.
\end{enumerate}
\end{theorem}
\begin{proof}
By Corollary \ref{col:unique} it follows that
\begin{equation*}
\norm{\mu}(B(u))=0.
\end{equation*}
Suppose that \ref{i:continuity}) holds true. Then
\begin{equation*}
\norm{\pi}\big(B(u)\times \mathbb{R}^n \big)=0.
\end{equation*}
Let 
\begin{equation*}
I=\big\{(x,y)\in\mathbb{R}^n\times\mathbb{R}^n\big|\norm{u(x)-u(y)}=\norm{x-y}\big\}.
\end{equation*}
By Theorem \ref{thm:ae}, $\norm{\pi}(I^c)=0$. 
Thus $\pi$ is concentrated on the set 
\begin{equation*}
C=I\cap B(u)^c\times \mathbb{R}^n.
\end{equation*}
Suppose that $(x,y)\in C$.
Then, as $A$ is a transport set, by the definition of $B(u)$,
\begin{equation}\label{eqn:equiv}
x\in A \text{ if and only if }y\in A.
\end{equation}
Let $\eta=\pi|_{A\times A}$. 
To prove \ref{i:transport}), it is enough to show that $\eta$ is an optimal transport and that  
\begin{equation*}
\eta\in\Gamma(\mu|_{A}).
\end{equation*}
For this, let $D\subset \mathbb{R}^n$ be any Borel set. Using the fact that $\pi\in\Gamma(\mu)$ and the fact that $\norm{\pi}(C^c)=0$ and (\ref{eqn:equiv}),  we have
\begin{equation*}
\begin{aligned}
&\mu(A\cap D)=\int_{\mathbb{R}^n\times\mathbb{R}^n}\Big(\mathbf{1}_{A\cap D}(x)-\mathbf{1}_{A\cap D}(y)\Big)d\pi(x,y)=\\
&=\int_{\mathbb{R}^n\times\mathbb{R}^n}\mathbf{1}_{A\times A}(x,y)\Big(\mathbf{1}_D(x)-\mathbf{1}_D(y)\Big)d\pi(x,y)=\\
&=\int_{\mathbb{R}^n\times\mathbb{R}^n}\Big(\mathbf{1}_D(x)-\mathbf{1}_D(y)\Big)d\eta(x,y)=\mathrm{P}_1\eta(D)-\mathrm{P}_2\eta(D).
\end{aligned}
\end{equation*}
It follows that $\pi_0|_{A\times  A}\in\Gamma(\mu|_A)$.
Then
\begin{equation}\label{eqn:comput}
\int_{A}\langle u,d\mu\rangle =\int_{\mathbb{R}^n\times\mathbb{R}^n} \mathbf{1}_I(x,y)\Big\langle \mathbf{1}_A(x) u(x)-\mathbf{1}_A(y)u(y),d\pi(x,y)\Big\rangle .
\end{equation}
Therefore, by (\ref{eqn:equiv}),
\begin{equation*}
\int_{A}\langle u_0,d\mu\rangle
=\int_{\mathbb{R}^n\times\mathbb{R}^n}\mathbf{1}_{A\times A}(x,y)\Big\langle u(x)-u(y),d\pi(x,y)\Big\rangle .
\end{equation*}
By condition \ref{i:borel}) of Theorem \ref{thm:ae} we see that
\begin{equation*}
\int_{A}\langle u,d\mu\rangle=\int_{A\times A}\norm{x-y}d\norm{\pi}(x,y).
\end{equation*}
Theorem \ref{thm:ae}, condition \ref{i:eren}), tells us that $\pi|_{A\times A}$ is an optimal transport and $u$ is an optimal potential.

Condition \ref{i:dezintegracja}) follows from \ref{i:desinteg}) readily. 

\end{proof}

\section{Counterexample}\label{sec:counter}

We shall now provide necessary tools for the aforementioned counterexample.

\begin{lemma}\label{lem:attain}
Let $X\subset\mathbb{R}^n$ be a compact set. Suppose that $(\mu_k)_{k=1}^{\infty}\subset\mathcal{W}_1(X,\mathbb{R}^m)$ converges weakly* to a measure $\mu_0\in\mathcal{W}_1(X,\mathbb{R}^m)$, i.e. for any 
continuous function $g\colon X\to\mathbb{R}^m$ we have
\begin{equation*}
\lim_{k\to\infty}\int_X\langle g,d\mu_k\rangle=\int_X\langle g,d\mu_0\rangle.
\end{equation*}
Suppose that $(u_k)_{k=1}^{\infty}\in\mathcal{L}(X,\mathbb{R}^m)$ are optimal potentials of $\mu_k$ respectively
and that $u_k$ converge uniformly to $u_0\colon X\to\mathbb{R}^m$. Then $u_0$ is an optimal potential of $\mu_0$.
\end{lemma}
\begin{proof}
By the assumption, for any continuous map $g\colon X\to\mathbb{R}^m$ we have
\begin{equation*}
\lim_{k\to\infty}\int_X\langle g,d\mu_k-\mu_0)\rangle=0.
\end{equation*}
By the Banach-Steinhaus theorem, the sequence $(\mu_k)_{k=1}^{\infty}$ is bounded in the total variation norm.
Hence, by uniform convergence,
\begin{equation*}
\lim_{k\to\infty}\int_X\langle u_k-u_0,d\mu_k\rangle=0.
\end{equation*}
It follows that 
\begin{equation*}
\int_X\langle u_k,d\mu_k\rangle=\int_X\langle u_0,d\mu_k\rangle+\int_X\langle u_k-u_0,d\mu_k\rangle
\end{equation*}
converges to $\int_X\langle u_0,d\mu_0\rangle$.
Thefefore for any $1$-Lipschitz map $h\colon X\to\mathbb{R}^m$ we have
\begin{equation*}
\int_X\langle h,d\mu_0\rangle\leq \int_X\langle u_0,d\mu_0\rangle.
\end{equation*}
\end{proof}

\begin{lemma}\label{lem:non-trivial}
Let $m\leq n$. Let $\mu\in\mathcal{W}_1(\mathbb{R}^n,\mathbb{R}^m)$ and let $u$ be an optimal potential. Suppose that there exists an optimal transport $\pi$ for $\mu$ or that any transport set for $u$ is of $\mu$ measure zero. Let $A$ be the union of all leaves of dimension at least one. Then
\begin{equation*}
\norm{\mu}(A^c)=0.
\end{equation*}
\end{lemma}
\begin{proof}
We know that $A$ is a Borel set. Suppose that there exists an optimal transport $\pi$ for $\mu$. By Theorem \ref{thm:ae}, $\pi$ is supported on the set
\begin{equation*}
I=\big\{(x,y)\in\mathbb{R}^n\times\mathbb{R}^n | \norm{u(x)-u(y)}=\norm{x-y}\big\}.
\end{equation*}
As $\mu=\mathrm{P}_1\pi-\mathrm{P}_2\pi$, for any Borel set $B\subset A^c$, we have
\begin{equation*}
\mu(B)=\pi(B\times\mathbb{R}^n)-\pi(\mathbb{R}^n\times B)=0,
\end{equation*}
for if $B\subset A^c$, then 
\begin{equation*}
B\times\mathbb{R}^n\cap I\subset \{(x,x)|x\in\mathbb{R}^n\}\text{ and }\mathbb{R}^n\times B \cap I\subset \{(x,x)|x\in\mathbb{R}^n\}.
\end{equation*}
Suppose now that any transport set for $u$ is of $\mu$ measure zero. Observe that any Borel set $B\subset A^c$ is a transport set. The conclusion follows.
\end{proof}

\begin{theorem}\label{thm:nonzero}
There exists an absolutely continuous measure $\mu\in\mathcal{W}_1(\mathbb{R}^n,\mathbb{R}^m)$ for which there is no optimal transport $\pi$ such that
\begin{equation*}
\mathrm{P}_1\norm{\pi}\ll\norm{\mu}.
\end{equation*}
Moreover, there exists a transport set associated with the optimal potential of $\mu$ with non-zero measure $\mu$. 
\end{theorem}
\begin{proof}
%

Choose any $v_1,\dotsc,v_{m+1}\in\mathbb{R}^m$ such that
\begin{equation*}
\sum_{i=1}^{m+1}v_i=0
\end{equation*}
and such that the kernel of the map 
\begin{equation*}
\mathbb{R}^{m+1}\ni(t_1,\dotsc ,t_{m+1})\mapsto\sum_{i=1}^{m+1}t_iv_i\in\mathbb{R}^m
\end{equation*}
is $\mathbb{R}(1,\dotsc,1)$. For $\epsilon>0$ set
\begin{equation*}
\mu_{\epsilon}=\frac1{\lambda(B(0,\epsilon)}\sum_{i=1}^{m+1} \lambda|_{B(x_i,\epsilon)}v_i,
\end{equation*}
where $x_1,\dotsc,x_{m+1}\in\mathbb{R}^n$ are pairwise distinct points to be specified later. Here $\lambda$ denotes the Lebesgue measure on $\mathbb{R}^n$. 
Then $\mu_{\epsilon}\in\mathcal{W}_1(\mathbb{R}^n,\mathbb{R}^m)$. Suppose that there exist optimal transports $\pi_k\in\Gamma(\mu_{\epsilon_k})$ such that
\begin{equation*}
\mathrm{P}_1\norm{\pi_k}\ll\norm{\mu_{\epsilon_k}}.
\end{equation*}
where $(\epsilon_k)_{k=1}^{\infty}$ is some sequence converging to zero. Then by Theorem \ref{thm:balance} we have
\begin{equation*}
\mu_{\epsilon_k}(A_k)=0
\end{equation*}
for any transport set $A_k$ of $u_k$, where $u_k\colon\mathbb{R}^n\to\mathbb{R}^m$ is an optimal potential of $\mu_{\epsilon_k}$.
For $k\in\mathbb{N}$ and $i=1,\dotsc,m+1$ consider the union $N_{ik}$ of all non-trivial leaves (i.e. of dimension at least one) that intersect $B(x_i,\epsilon_k)$. Then $N_{ik}$ is a transport set.
 Indeed, its Borel measurability follows from measurability of the map $\mathcal{S}$, which is proven before.
Thus $\mu(N_{ik})=0$. Hence, 
\begin{equation}\label{eqn:sum}
\sum_{j=1}^{m+1}v_j\lambda(B(x_j,\epsilon_k)\cap N_{ik})=0.
\end{equation}
As $\mu_{\epsilon_k}$, by Lemma \ref{lem:non-trivial}, is concentrated on non-trivial leaves of $u_k$, we have for 
\begin{equation*}
\frac{\lambda(B(x_i,\epsilon_k)\cap N_{ik})}{\lambda(B(0,\epsilon_k))}v_i=\mu_{\epsilon_k}(B(x_i,\epsilon_k)\cap N_{ik})=\mu_{\epsilon_k}(B(x_i,\epsilon_k)=v_j.
\end{equation*}
By (\ref{eqn:sum}) and assumption on the vectors $v_1,\dotsc,v_{m+1}$
\begin{equation*}
\lambda(B(x_j,\epsilon_k)\cap N_{ik})=\lambda(B(0,\epsilon_k)) \text{ for all }j=1,\dotsc,m+1.
\end{equation*}
Thus we infer that for any $k\in\mathbb{N}$ and for all $r,s=1,\dotsc,m+1$, $r\neq s$, there exist points
\begin{equation*}
(x_{rs}^k,x_{sr}^k)\in B(x_r,\epsilon_k)\times B(x_s,\epsilon_k)
\end{equation*}
such that 
\begin{equation*}
\norm{u_k(x_{rs}^k)-u_k(x_{sr}^k)}=\norm{x_{rs}^k-x_{sr}^k}.
\end{equation*}
Using Arz\`ela-Ascoli theorem and passing to a subsequence we may assume that $u_k$ converge locally uniformly to some $1$-Lipschitz map $u_0$. Observe now that 
\begin{equation*}
x_{rs}^k\text{ converges to }x_r\text{ for all }r,s=1,\dotsc,m+1.
\end{equation*}
Thus, by the locally uniform convergence, $u_0$ is an isometry on $\{x_1,\dotsc,x_{m+1}\}$. Observe that 
\begin{equation*}
\mu_{\epsilon_k}\text{ converges weakly* to }\mu_0=\sum_{i=1}^{m+1}\delta_{x_i}v_i.
\end{equation*}
Now Lemma \ref{lem:attain} tells us that $u_0$ is an optimal potential of $\mu_0$.

Suppose now that points $x_1,\dotsc,x_{m+1}$ are such that for $i\neq j$, $i,j=1,\dotsc,m$,
\begin{equation}\label{eqn:ineq}
\Big\langle\frac{x_i-x_{m+1}}{\norm{x_i-x_{m+1}}} , \frac{x_j-x_{m+1}}{\norm{x_j-x_{m+1}}}\Big\rangle< \Big\langle\frac{v_i}{\norm{v_i}},\frac{v_j}{\norm{v_j}}\Big\rangle.
\end{equation}
Then if we define $h\colon\{x_1,\dotsc,x_{m+1}\}\to\mathbb{R}^m$ by 
\begin{equation*}
h(x_{m+1})=0\text{, }h(x_i)=\norm{x_i-x_{m+1}}\frac{v_i}{\norm{v_i}}\text{ for }i=1,\dotsc,m,
\end{equation*}
then $h$ is $1$-Lipschitz. By the Kirszbraun theorem we may assume that $h$ is defined on the whole plane. Moreover for 
\begin{equation*}
\pi=\sum_{i=1}^{m+1}v_i\delta_{(x_i,x_{m+1})}
\end{equation*}
we have
\begin{equation*}
\mathrm{P}_1\pi-\mathrm{P}_2\pi=\mu_0
\end{equation*}
and 
\begin{equation*}
\pi=\sum_{i=1}^{m}\frac{h(x_i)-h(x_{m+1})}{\norm{x_i-x_{m+1}}}\norm{v_i}\delta_{(x_i,x_{m+1})}
\end{equation*}
Theorem \ref{thm:ae} yields that $h$ is an optimal potential and $\pi$ is an optimal transport. It follows that 
\begin{equation*}
\norm{\mu_0}_{\mathcal{W}_1(\mathbb{R}^2,\mathbb{R}^2)}=\sum_{i=1}^m\norm{v_i}\norm{x_i-x_{m+1}}.
\end{equation*}
Theorem \ref{thm:ae} tells us that also
\begin{equation*}
\pi=\sum_{i=1}^{m}\frac{u_0(x_i)-u_0(x_{m+1})}{\norm{x_i-x_{m+1}}}\norm{v_i}\delta_{(x_i,x_{m+1})}
\end{equation*}
As $u_0$ is an isometry on $\{x_1,\dotsc,x_{m+1}\}$, It follows that
\begin{equation*}
\norm{h(x_1)-h(x_2)}=\norm{x_1-x_2}
\end{equation*}
which is not true, as the inequality in (\ref{eqn:ineq}) is strict.
The obtained contradiction shows that there is no such sequence $(\epsilon_k)_{k=1}^{\infty}$, i.e. there exists $\epsilon_0>0$ such that for all $\epsilon\in (0,\epsilon_0)$ there is no optimal transport with absolutely continuous marginals.
%
\end{proof}
The following theorem bases on the same idea as the former one. 
Note that we do not require below that the norms on $\mathbb{R}^n$ and on $\mathbb{R}^m$ are Euclidean. The leaves and transport sets are defined as in the Euclidean case.

\begin{theorem}
Let $m\leq n$. Suppose that the norm on $\mathbb{R}^m$ is strictly convex. Suppose that $\mathcal{F}$ is a uniformly closed subset of $1$-Lipschitz maps of $\mathbb{R}^n$ to $\mathbb{R}^m$. Suppose that $\mathcal{F}$ has the property that for any absolutely continuous measure $\mu\in\mathcal{W}_1(\mathbb{R}^n,\mathbb{R}^m)$ and any $u_0\in\mathcal{F}$ such that
\begin{equation}\label{eqn:maxil}
\int_{\mathbb{R}^n}\langle u_0,d\mu\rangle=\sup\Big\{\int_{\mathbb{R}^n}\langle u,d\mu\rangle\big| u\in\mathcal{F}\Big\}
\end{equation}
we have $\mu(A)=0$ for any transport set of $u_0$. Then either $m=1$ or $m>1$ and then any $u\in\mathcal{F}$ is affine, $m=n$ and $\mathbb{R}^n$ and $\mathbb{R}^n$ with the considered norms are isometric.

Moreover, for any $\mathcal{F}$-optimal potential is an isometry on a maximal subspace $V\subset\mathbb{R}^n$, so that for any absolutely continuous $\mu$, there is a linear subspace $V\subset\mathbb{R}^n$ such that 
\begin{equation}\label{eqn:condition}
\mu(\{x\in\mathbb{R}^n| P^{\perp}x\in A\})=0\text{ for any Borel set }A\subset V^{\perp}.
\end{equation}
Here $P^{\perp}$ denotes the orthogonal projection onto the orthogonal complement $V^{\perp}$ of $V$.

If any $\mathcal{F}$-optimal potential is an isometry on a maximal subspace $V\subset\mathbb{R}^n$ such that (\ref{eqn:condition}) holds true , then $\mu(A)=0$ for any transport set of its $\mathcal{F}$-optimal potential.
\end{theorem}

Above, if $\mu\in\mathcal{M}_0(\mathbb{R}^n,\mathbb{R}^m)$ and a map $u_0\in\mathcal{F}$ is such that (\ref{eqn:maxil}) holds true, then we call $u_0$ an $\mathcal{F}$-optimal potential of $\mu$.

\begin{proof}
Suppose that $m>1$. Choose any pairwise different $x_1,x_2,x_3\in\mathbb{R}^n$ and $v_1,v_2,v_3\in\mathbb{R}^m$ in general position such that $\sum_{i=1}^3v_i=0$. Let 
\begin{equation*}
\nu_0=\sum_{i=1}^3v_i\delta_{x_i}.
\end{equation*}
Then $\nu_0\in\mathcal{M}_0(\mathbb{R}^n,\mathbb{R}^m)$. For $\epsilon>0$ let 
\begin{equation*}
\nu_{\epsilon}=\frac1{\lambda(B(0,\epsilon)}\sum_{i=1}^3v_i\lambda|_{B(x_i,\epsilon)}
\end{equation*}
Choose an $\mathcal{F}$-optimal potentials $u_{\epsilon}$ for $\nu_{\epsilon}$ respectively.
Observe that $\nu_{\epsilon}(B_{\epsilon})=0$ for any Borel set consisting of zero-dimensional leaves of $u_{\epsilon}$. Whence, $\nu_{\epsilon}$ is concentrated on at least one-dimensional transport sets of $u_{\epsilon}$.
Let $N_{i\epsilon}$ denote the union of all non-trivial leaves that intersect $B(x_i,\epsilon)$ for $i=1,2,3$ and $\epsilon>0$. By compactness of $B(x_i,\epsilon)$ and by the assumption on transport sets
\begin{equation*}
N_{i\epsilon}=\Big\{x\in\mathbb{R}^n\setminus B(x_i,\epsilon)\big| \sup\Big\{\frac{\norm{u(x)-u(y)}}{\norm{x-y}}\big| y\in B(x_i,\epsilon)\Big\}=1\Big\}\cup B(x_i,\epsilon)\setminus N
\end{equation*}
Here $N$ is a set of points in $B(x_i,\epsilon)$ that belong to a zero-dimensional leaves,
\begin{equation*}
N=\Big\{x\in B(x_i,\epsilon)\Big|\frac{\norm{u(x)-u(y)}}{\norm{x-y}}<1\text{ for any }y\in\mathbb{R}^n\Big\}.
\end{equation*}
The map $x\mapsto \sup\big\{\frac{\norm{u(x)-u(y)}}{\norm{x-y}}\big| y\in K\big\}$ is lower-semicontinuous for any set $K\subset\mathbb{R}^n$. Hence $N$ is Borel measurable by $\sigma$-compactness of $\mathbb{R}^n$ and so is $N_{i\epsilon}$.
By the assumption,
\begin{equation*}
\nu_{\epsilon}(N_{i\epsilon})=0,
\end{equation*}
which implies, as before, that 
\begin{equation*}
\norm{u_{\epsilon}(x^{\epsilon}_{rs})-u_{\epsilon}(x^{\epsilon}_{sr})}=\norm{x^{\epsilon}_{rs}-x^{\epsilon}_{sr}}
\end{equation*}
for some points 
\begin{equation*}
(x^{\epsilon}_{rs},x^{\epsilon}_{sr})\in B(x_r,\epsilon)\times B(x_s,\epsilon).
\end{equation*}
By the Arz\`ela-Ascoli theorem and passing to a subsequence we may assume that $u_{\epsilon}$ converges locally uniformly to some $u_0\in\mathcal{F}$, which is an $\mathcal{F}$-optimal potential of $\nu_0$ by Lemma \ref{lem:attain}.
By the uniform convergence we infer that $u_0$ is isometric on $\{x_1,x_2,x_3\}$. Let now $x_2=tx_1+(1-t)x_3$ for some $t\in (0,1)$. Then any $1$-Lipschitz map $f$ that is isometric on $\{x_1,x_2,x_3\}$ satisfies 
\begin{equation}\label{eqn:cc}
f(tx_1+(1-t)x_3)=tf(x_1)+(1-t)f(x_3).
\end{equation}
Indeed, by the assumption, 
\begin{equation*}
\norm{f(x_2)-f(x_1)}= (1-t)\norm{x_3-x_1}\text{ and }\norm{f(x_3)-f(x_2)}= t\norm{x_3-x_1}.
\end{equation*}
As $\norm{f(x_3)-f(x_1)}=\norm{x_3-x_1}$ it follows that we have equality in the triangle inequality
\begin{equation*}
\norm{f(x_3)-f(x_1)}\leq\norm{f(x_2)-f(x_1)}+\norm{f(x_3)-f(x_2)}.
\end{equation*}
By the strict convexity it follows that there is $\lambda>0 $ such that 
\begin{equation*}
f(x_2)-f(x_1)=\lambda(f(x_3)-f(x_1)). 
\end{equation*}
Taking the norms we arrive at (\ref{eqn:cc}).
A function that satisfies (\ref{eqn:cc}) may be extended to $\mathbb{R}^n$ to an affine map that is isometric on $\{x_1,x_2,x_3\}$ and with derivative of operator norm at most one. Indeed, it is enough to show that if $f\colon \mathbb{R}w\to\mathbb{R}z$ for some vectors $w,z$ is of norm at most one, that there exists a linear extension of $f$ with the same norm. This follows by the Hahn-Banach theorem.
We infer that
\begin{equation*}
\sum_{i=1}^3\langle u_0(x_i),v_i\rangle\ \leq \sup\Big\{\sum_{i=1}^3\langle f(x_i),v_i\rangle\big| f\text{ is linear and }\norm{f}\leq 1\Big\}
\end{equation*} 
As the set of vectors $v_1,v_2,v_3$ that sum up to zero and are in general position is dense in the set of vectors $v'_1,v'_2,v'_3$ that sum up to zero and by the fact that $u_0$ is an $\mathcal{F}$-optimal potential for $\nu_0$ we conclude that for any $u\in\mathcal{F}$ and any vectors $v_1,v_2,v_3$ that sum up to zero there is
\begin{equation*}
\sum_{i=1}^3\langle u(x_i),v_i\rangle\ \leq \sup\Big\{\sum_{i=1}^3\langle f(x_i),v_i\rangle\big|f\text{ is linear and }\norm{f}\leq 1\Big\}
\end{equation*} 
Take now $v_2=v$, $v_1=-tv$ and $v_3=-(1-t)v$ with $t\in (0,1)$ as above and any $v\in\mathbb{R}^m$. We infer that
\begin{equation*}
\langle u(x_2)-tu(x_1)-(1-t)u(x_3) ,v\rangle\leq 0.
\end{equation*}
As this holds for any $v$ we infer that $u$ is affine. 
 If $u$ is affine then there exists a subspace $V\subset\mathbb{R}^n$, possibly trivial, i.e. $V=\{0\}$, such that any set of the form 
\begin{equation*}
\{x\in\mathbb{R}^n| P^{\perp}x\in A\}
\end{equation*} 
for a Borel measurable set $A\subset V^{\perp}$ is a transport set of $u$. Here $P^{\perp}$ denotes a projection onto a complement $V^{\perp}$ of $V$.  Indeed, let $V\subset\mathbb{R}^n$ be a maximal subspace such that $u|_V$ is an isometry.  Suppose that $V$ is not a leaf of $u$. Then there exists $y\notin V$ such that for all $x\in V$
\begin{equation*}
\norm{u(y)-u(x)}=\norm{y-x}.
\end{equation*}
It follows that for all non-zero $\lambda\in\mathbb{R}$
\begin{equation*}
\Big\lVert u(y)-u\Big(\frac{x}{\lambda}\Big)\Big\rVert=\Big\lVert y-\frac{x}{\lambda}\Big\rVert
\end{equation*}
for all $x\in V$. Hence for all $\lambda\in\mathbb{R}$ we have $\norm{u(\lambda y)-u(x)}=\norm{\lambda y-x}$. As $u$ is affine, it is also an isometry on $V+\mathbb{R}y$. This contradiction shows that $V$ is a leaf of $u$. 

We shall now provide an example of a vector measure $\mu$ such that for any proper subspace $V$ and any $x_0$ there is $c>0$ such that 
\begin{equation}\label{eqn:mum}
\mu\Big(\big\{x\in\mathbb{R}^n| \norm{P^{\perp}(x-x_0)}\leq c\big\}\Big)\neq 0.
\end{equation}
Choose any $x_1,\dotsc,x_{m+1}\in\mathbb{R}^n$ in general position. Let $\epsilon>0$ be a number such that any $y_i\in B(x_i,\epsilon)$, $i=1,\dotsc,m+1$ are in general position. Choose vectors $v_1,\dotsc,v_{m+1}\in\mathbb{R}^m$ that add up to zero and are in general position. Let 
\begin{equation*}
\mu=\sum_{i=1}^{m+1}v_i\lambda|_{B(x_i,\epsilon)},
\end{equation*}
where $\lambda$ denotes the Lebesgue measure. Choose any proper affine subspace $V\subset\mathbb{R}^n$. Then $V$ intersects at most $m$ of the balls $B(x_i,\epsilon)$, $i=1,\dotsc,m+1$. So does the set 
\begin{equation*}
\big\{x\in\mathbb{R}^n| \norm{P^{\perp}(x-x_0)}\leq c\big\}
\end{equation*}
provided that $c>0$ is sufficiently small. Thus (\ref{eqn:mum}) follows. We have shown that any $\mathcal{F}$-optimal potential of $\mu$ has to be an isometry. Hence $m= n$.

To prove the last part of the theorem, it is enough to observe that $V$ and its translates are the only leaves of an $\mathcal{F}$-optimal potential. This holds true, as these sets are maximal sets such that restriction of $u$ to them is isometric and they cover $\mathbb{R}^n$.
\end{proof}

\section{Curvature-dimension condition}\label{sec:curv}

In the current section we recall the notion of the curvature-dimension condition $CD(\kappa,n)$. We shall say that an $n$-dimensional Riemannian manifold $\mathcal{M}$ satisfies the $CD(\kappa,n)$ condition provided that the Ricci tensor $Ric_M$ is bounded below by the Riemannian metric tensor $g$, i.e.
\begin{equation*}
Ric_{\mathcal{M},p}(v,v)\geq\kappa g_p(v,v)\text{ for any }p\in \mathcal{M}\text{ and any }v\in T_p\mathcal{M}.
\end{equation*}
We shall study weighted Riemannian manifolds, which are triples $(\mathcal{M},d,\mu)$, where $d$ is the Riemannian metric on $\mathcal{M}$ and $\mu$ is a measure on $\mathcal{M}$ with smooth positive density $e^{-\rho}$ with respect to the Riemannian volume. The generalised Ricci tensor of the weighted Riemannian manifold is defined by the formula
\begin{equation*}
Ric_{\mu}=Ric_\mathcal{M}+D^2\rho,
\end{equation*}
where $D^2\rho$ is the Hessian of smooth function $\rho$. The generlised Ricci tensor with parameter $N\in (-\infty,1)\cup [n,\infty]$ is defined by the formula
\begin{equation*}
Ric_{\mu,N}=\begin{cases}
    Ric_{\mu}(v,v)-\frac{D\rho(v)^2}{N-n},& \text{if } N>n\\
    Ric_{\mu}(v,v)&\text{if }N=\infty\\
    Ric_{\mathcal{M}}(v,v)    &\text{if }N=n\text{ and }\rho\text{ is constant.}
\end{cases}
\end{equation*}
Note that if $N=n$, then $\rho$ is required to be a constant function.

\begin{definition}
For $\kappa\in\mathbb{R}$ and $N\in (-\infty,1)\cup [n,\infty]$ we say that $(\mathcal{M},d,\mu)$ satisfies the curvature-dimension condition $CD(\kappa,N)$ if 
\begin{equation*}
Ric_{\mu,N}(v,v)\geq \kappa g(v,v)\text{ for all }x\in \mathcal{M}\text{ and all }v\in T_p\mathcal{M}.
\end{equation*}
\end{definition}

We refer the reader to \cite{Bakry} and \cite{Ledoux} for background on the curvature-dimension condition. 
In all cases we consider in this article it will always hold that $Ric_{\mathcal{M}}=0$. 

Let us recall a lemma from \cite{Klartag} that we shall need in what follows.

\begin{lemma}\label{lem:tri}
Let $a,b\in\mathbb{R}$, $b>0$ and $a\notin [-b,0]$. Then
\begin{equation*}
\frac{x^2}{a}+\frac{y^2}{b}\geq\frac{(x-y)^2}{a+b}
\end{equation*}
for all $x,y\in\mathbb{R}$.
\end{lemma}
\begin{proof}
We use the inequality
\begin{equation*}
\frac{\abs{a}}{\abs{b}}x^2\pm 2xy+\frac{\abs{b}}{\abs{a}}y^2\geq 0.
\end{equation*}
From this we see that
\begin{equation*}
\frac{x^2}{a}+\frac{y^2}{b}-\frac{(x-y)^2}{a+b}= \frac1{a+b}\Big(\frac{b}{a}x^2+2xy+\frac{a}{b}y^2\Big)\geq 0
\end{equation*}
whenever $b>0$ and $a\notin [-b,0]$.
\end{proof}

Let us also recall a formulae for differentiation of matrices. If $R(t)= \log\abs{\det A_t}$ and $A_t$ is differentiable in $t\in\mathbb{R}$, then
\begin{equation}\label{eqn:formula}
\frac{dR}{dt}(s)=tr\Big(A_s^{-1}\frac{ dA_t}{dt}(s)\Big).
\end{equation}
Moreover
\begin{equation}\label{eqn:formula2}
\frac{d^2R}{dt^2}(s)=tr\Big(A_s^{-1}\frac{ d^2A_t}{dt^2}(s)\Big)-tr\Bigg(\Big(A_s^{-1}\frac{ dA_t}{dt}(s)\Big)^2\Bigg).
\end{equation}

We should also need the following version of the Whitney extension theorem (see \cite{Whitney} or \cite{Stein}).

\begin{theorem}
Let $A\subset\mathbb{R}^n$ be an arbitrary set, let $f\colon A\to\mathbb{R}$ and $V\colon A\to\mathbb{R}^n$. Suppose that there exists $M\in\mathbb{R}$ such that for all $x,y\in A$
\begin{align*}
&\abs{f(x)}\leq M, \norm{V(x)}\leq M,\\
& \norm{V(x)-V(y)}\leq M\norm{x-y},\\
&\abs{f(x)+\langle V(x),y-x\rangle-f(y)}\leq M\norm{x-y}^2.
\end{align*}
Then there exists a differentiable function $\tilde{f}\colon\mathbb{R}^n\to\mathbb{R}$ with locally Lipschitz derivative such that 
\begin{equation*}
\tilde{f}(x)=f(x), Df(x)(y)=\langle V(x),y\rangle\text{ for all }x\in A\text{ and all }y\in\mathbb{R}^n.
\end{equation*}
\end{theorem}

Assume that we have a measure $\mu$ on $\mathcal{M}=\mathbb{R}^n$ such that $(\mathcal{M},\norm{\cdot},\mu)$ satisfies the curvature-dimension condition $CD(\kappa,N)$. Let $u\colon\mathbb{R}^n\to\mathbb{R}^m$ be a $1$-Lipschitz map. We want to show that for $\nu$-almost every leaf $\mathcal{S}\in CC(\mathbb{R}^n)$ of dimension $m$  the conditional measure $\mu_{\mathcal{S}}$ is such that $(\mathrm{int}\mathcal{S},\norm{\cdot},\mu_{\mathcal{S}})$ satisfies the curvature-dimension condition $CD(\kappa,N)$. Here $\nu$ is the push-forward measure of $\mu$ with respect to the map $\mathcal{S}\colon\mathbb{R}^n\to CC(\mathbb{R}^n)$.

\begin{theorem}\label{thm:discd}
Let $N\in (-\infty,1)\cup [n,\infty]$ andl let $\kappa\in\mathbb{R}$. Let $u\colon\mathbb{R}^n\to\mathbb{R}^m$ be a $1$-Lipschitz map with respect to the Euclidean norms. Let $\mu$ be a Borel measure on $\mathbb{R}^n$ such that $(\mathbb{R}^n,\norm{\cdot},\mu)$ satisfies the curvature-dimension condition $CD(\kappa,N)$. Then there exists a map $\mathcal{S}\colon\mathbb{R}^n\to CC(\mathbb{R}^n)$ such that for $\lambda$-almost every $x\in\mathbb{R}^n$ the set $\mathcal{S}(x)$ is a maximal closed convex set in $\mathbb{R}^n$ such that $u|_{\mathcal{S}(x)}$ is an isometry. Moreover, there exist a Borel measure on $CC(\mathbb{R}^n)$ and Borel measures $\mu_{\mathcal{S}}$ such that 
\begin{equation*}
\mathcal{S}\mapsto \mu_{\mathcal{S}}(A)\text{ is }\nu\text{-measurable for any Borel set }A\subset\mathbb{R}^n
\end{equation*}
and for $\nu$-almost every $\mathcal{S}$ we have $\mu_{\mathcal{S}}((\mathrm{int}\mathcal{S})^c)=0$,
and for any $A\subset\mathbb{R}^n$
\begin{equation*}
\mu(A)=\int_{CC(\mathbb{R}^n)} \mu_{\mathcal{S}}(A)d\nu(\mathcal{S}).
\end{equation*}
Moreover, for $\nu$-almost every leaf $\mathcal{S}$ of dimension $m$, the measure $\mu_{\mathcal{S}}$ is such that $(\mathrm{int}\mathcal{S},\norm{\cdot},\mu_{\mathcal{S}})$ satisfies the $CD(\kappa,N)$ condition.
\end{theorem}
\begin{proof}
Let us fix a cluster $T_{pij}$. Note that by Theorem \ref{thm:disinteg} the density of the conditional measures $\mu_{\mathcal{S}}$ for a leaf of dimension $m$ is equal to
\begin{equation*}
\frac{d\mu_{\mathcal{S}}}{d\mathcal{H}_m}=c (J_nF)\circ Ge^{-\rho}\mathbf{1}_{\mathcal{S}},
\end{equation*}
where $c$ is a positive normalising constant. Here $J_nF$ denotes the Jacobian of $F$.
Recall that $F,G$ are given by the formulae
\begin{equation*}
F(a,b)=v(a)+Du(v(a))^*(b)
\end{equation*}
and
\begin{equation*}
G(x)=(w(z),u(x)-u(z)),
\end{equation*}
where $w\colon \mathbb{R}^n\to\mathbb{R}^{n-m}$ and $v\colon \mathbb{R}^{n-m}\to\mathbb{R}^n$ are maps from Lemma \ref{lem:cover}, see also Lemma \ref{lem:efge} for details. Let us recall that $v(a)\in S_p^i$ for all $a\in\Lambda$. Here 
\begin{equation*}
\Lambda=\{a\in\mathbb{R}^{n-m}|(a,0)\in G(\mathrm{int}T_{pij})\}.
\end{equation*}
 It follows by the definition of $S_p$ that $u(v(a))=p$ for all $a\in\Lambda$. Recall, that by Lemma \ref{lem:diff}, $u$ is differentiable in $\mathrm{int}T_{pij}$. Thus, if $b\in\mathbb{R}^m$ is such that pair $(a,b)\in G(\mathrm{int}T_{pij})$ then 
\begin{equation}\label{eqn:compo}
Du(v(a))Dv(a)=0\text{ for almost every }a\in\Lambda.
\end{equation}
Note that, by Corollary \ref{col:strength}, see also Lemma \ref{lem:efge}, on $T_{pij}^{\lambda,\rho}$, $Du$ is Lipschitz. By the Whitney extension theorem there exists a differentible map $\tilde{u}$ with Lipschitz derivative on $\mathbb{R}^n$ that coincides with $u$ on $T_{pij}^{\lambda,\rho}$ and such that $D\tilde{u}=Du$ on $T_{pij}^{\lambda,\rho}$. By a lemma from \cite[Lemma 3.12]{Klartag}, the second derivative of $\tilde{u}$ exists almost everywhere and is symmetric, in the sense that the second derivative of any of its components is symmetric. We will abuse the notation and assume that $u$ has Lipschitz derivative. 

The derivative of $F$ is equal to
\begin{equation*}
DF(a,b)=[Dv(a)+D^2u(v(a))^*(Dv(a)(\cdot))(b),Du(v(a))^*].
\end{equation*}
Note that for any vectors $z\in\mathbb{R}^{n-m}$ and $w\in\mathbb{R}^m$ the derivatives $Dv(a)z$ and $Du(v(a))^*w$ are orthogonal.
Indeed, by (\ref{eqn:compo}),
\begin{equation*}
\Big\langle Du(v(a))^*(w),Dv(a)(z)\Big\rangle=\Big\langle w, Du(v(a))Dv(a)(z)\Big\rangle=0.
\end{equation*}
Let $P$ denote the orthogonal projection onto the tangent space of the leaf containing $v(a)$. Then $Du(v(a))=TP$, see Lemma \ref{lem:diff} and Lemma \ref{lem:efge}.
Let $P^{\perp}$ denote the orthogonal projection onto its orthogonal complement. Then
\begin{equation*}
DF(a,b)=[Dv(a)+D^2u(v(a))^*(P^{\perp}Dv(a)(\cdot))(b),Du(v(a))^*].
\end{equation*}
Therefore, as $Du(v(a))^*$ is isometric, we have
\begin{equation*}
\abs{\det(DF(a,b))}=\Big\lvert\det\Big(Dv(a)+P^{\perp}D^2u(v(a))^*(P^{\perp}Dv(a)(\cdot))(b)\Big)\Big\rvert,
\end{equation*}
which is equal to 
\begin{equation*}
\abs{\det Dv(a)}\Big\lvert\det\Big(\mathrm{Id}+P^{\perp}D^2u(v(a))^*(P^{\perp}(\cdot))(b)\Big)\Big\rvert.
\end{equation*}
Note that 
\begin{equation*}
H(b)=\Big(\mathrm{Id} +P^{\perp}D^2u(v(a))^*(P^{\perp}(\cdot))(b)\Big)
\end{equation*}
is a linear operator on the image of $P^{\perp}$, which is of dimension $n-m$. Moreover it is symmetric and invertible for any $b$ such that $(a,b)\in G(\mathrm{int}T_{pij})$, as $F$ is bijection. Consider for some $b'\in\mathbb{R}^m$
\begin{equation*}
P^{\perp}D^2u(v(a))^*(P^{\perp}(\cdot))(b').
\end{equation*}
Let $A$ be such that
\begin{equation*}
P^{\perp}D^2u(v(a))^*(P^{\perp}(\cdot))(b')=A\Big(\mathrm{Id} +P^{\perp}D^2u(v(a))^*(P^{\perp}(\cdot))(b)\Big).
\end{equation*}
Then $A$ is conjugate to a symmetric operator of rank at most $n-m$, as
\begin{equation*}
H(b)^{-\frac12}AH(b)^{\frac12}=H(b)^{-\frac12}P^{\perp}D^2u(v(a))^*(P^{\perp}(\cdot))(b')H(b)^{-\frac12}.
\end{equation*}
In consequence, by the Cauchy-Schwarz inequality
\begin{equation}\label{eqn:bac}
(trA)^2\leq (n-m)tr(A)^2.
\end{equation}

Let $x=F(a,b)$ and note that any $v$ in the tangent space of $\mathcal{S}$ is of the form $v=Du(v(a))^*(b')$ for some $b'\in\mathbb{R}^m$.  Then
\begin{align*}
D\log\abs{\det DF\circ G}(x)(v)&=\frac{d}{dt}\log \abs{\det\big( DF(G(F(a,b)+tDu(v(a))^*(b')\big)}=\\
&=\frac{d}{dt}\log\abs{\det (DF(a,b+tb'))}=\frac{d}{dt}\log\det H(b+tb').
\end{align*}
Therefore, by (\ref{eqn:formula}) and (\ref{eqn:formula2}),
\begin{equation*}
D\log\abs{\det DF\circ G}(x)(v)=tr\Big(H(b)^{-1}P^{\perp}D^2u(v(a))^*(P^{\perp}(\cdot))(b')\Big)=trA
\end{equation*}
and
\begin{equation*}
D^2\log\abs{\det DF\circ G}(x)(v,v)=-tr\Big(H(b)^{-1}P^{\perp}D^2u(v(a))^*(P^{\perp}(\cdot))(b')\Big)^2=-tr(A^2).
\end{equation*}
By (\ref{eqn:bac}) and by Lemma \ref{lem:tri}, if $N\notin [m, n]$, then
\begin{align*}
-D^2 \log\abs{\det DF\circ G}(v,v)&=tr(A^2)\geq \\
&\geq\frac{1}{n-m}(trA)^2\geq \frac1{N-m}(D\rho (v)-tr A)^2-\frac{(D\rho(v))^2}{N-n}.
\end{align*}
Note that by the assumption for all $v\in\mathbb{R}^n$
\begin{equation*}
D^2{\rho}(v,v)-\frac{D\rho(v)^2}{N-n}\geq \kappa \norm{v}^2.
\end{equation*}
Thus for all $v$ in the tangent space of $\mathcal{S}$ there is
\begin{equation*}
D^2\rho(v,v)-D^2 \log\abs{\det DF\circ G}(v,v)-\frac{\big(D\rho(v)-D(\log\abs{\det DF\circ G})(v)\big)^2}{N-m}\geq \kappa \norm{v}^2.
\end{equation*}
We infer that $(\mathrm{int}\mathcal{S},\norm{\cdot},\mu_{\mathcal{S}})$ satisfies the curvature-dimension condition $CD(\kappa,N)$, provided that $N\notin[m,n]$.

If $N=n$, then $\rho$ is required to be a constant function, and thus in this case the inequality is also satisfied. If $N=\infty$, then the estimates are trivial.
\end{proof}
For the historical remarks on similar estimates we refer to \cite{Klartag}. 

\bibliographystyle{plain}
\bibliography{biblio}

\label{lastpage}
\end{document}